\documentclass[12pt]{article}
\usepackage[colorlinks,allcolors=blue]{hyperref}
\newcommand{\doi}[1]{\href{https://doi.org/#1}{doi:#1}}
\newcommand{\myurl}[1]{\href{#1}{#1}}
\usepackage[utf8]{inputenc}
\usepackage[width=16cm,height=21cm]{geometry}
\usepackage{amsmath,amssymb,amsthm,mathtools,booktabs}
\usepackage{cite}
\usepackage{stackengine}
\stackMath
\usepackage{verbatim} 

\newtheorem{thm}{Theorem}[section]
\newtheorem{prop}[thm]{Proposition}
\newtheorem{lem}[thm]{Lemma}
\newtheorem{cor}[thm]{Corollary}
\theoremstyle{definition}

\newtheorem{example}[thm]{Example}
\newtheorem{rem}[thm]{Remark}

\newtheorem{tempremark}{Remark for the reviewers}

\newcommand{\bC}{\mathbb{C}}

\newcommand{\bN}{\mathbb{N}}

\newcommand{\cT}{\mathcal{T}}
\newcommand{\cTmonic}{\mathcal{T}^{\operatorname{monic}}}
\newcommand{\cU}{\mathcal{U}}
\newcommand{\cV}{\mathcal{V}}
\newcommand{\cW}{\mathcal{W}}

\newcommand{\Partitions}{\mathcal{P}}
\newcommand{\coefT}{\tau}
\newcommand{\coefmonomialviaT}{\alpha}
\newcommand{\Catalan}{\textbf{C}}
\newcommand{\coefsumU}{\sigma}

\newcommand{\la}{\lambda}
\newcommand{\al}{\alpha}

\newcommand{\elemz}{\operatorname{ez}}
\newcommand{\homz}{\operatorname{hz}}
\newcommand{\schurz}{\operatorname{sz}}
\newcommand{\Elem}{\operatorname{E}}
\newcommand{\Hom}{\operatorname{H}}
\newcommand{\Epal}{\widetilde{\operatorname{E}}}
\newcommand{\Hpal}{\widetilde{\operatorname{H}}}
\newcommand{\powersumz}{\operatorname{pz}}

\newcommand{\rev}{\operatorname{rev}}
\newcommand{\id}{\operatorname{id}}
\newcommand{\schur}{\operatorname{s}}

\newcommand{\elem}{\operatorname{e}}
\renewcommand{\hom}{\operatorname{h}}
\newcommand{\powersum}{\operatorname{p}}
\newcommand{\eqdef}{\coloneqq}
\newcommand{\VandermondePol}{\operatorname{Van}}

\newcommand{\sympl}{\operatorname{sp}}
\newcommand{\orth}{\operatorname{o}}
\newcommand{\spz}{\operatorname{spz}}
\newcommand{\oz}{\operatorname{oz}}

\newcommand{\odd}{\operatorname{odd}}

\newcommand{\Sym}{\operatorname{Sym}}

\newcommand{\proddif}{\Omega}
\newcommand{\Phiodd}{\Phi^{\operatorname{odd}}}
\newcommand{\adjugate}{\operatorname{adj}}
\newcommand{\LR}{\operatorname{LR}}

\newcommand{\detmatr}[2]{\left|\begin{array}{#1}#2\end{array}\right|}

\usepackage{xcolor}

\title{
Symmetric polynomials in the symplectic alphabet\\
and their expression via Dickson--Zhukovsky variables}
\author{Per Alexandersson, Luis Angel Gonz\'{a}lez-Serrano,\\
Egor A. Maximenko, Mario Alberto Moctezuma-Salazar}

\begin{document}
\maketitle

\begin{abstract}
Given a symmetric polynomial $P$ in $2n$ variables,
there exists a unique symmetric polynomial $Q$ in $n$ variables such that
\[
P(x_1,\ldots,x_n,x_1^{-1},\ldots,x_n^{-1})
=Q(x_1+x_1^{-1},\ldots,x_n+x_n^{-1}).
\]
We denote this polynomial $Q$ by $\Phi_n(P)$
and show that $\Phi_n$ is an epimorphism of algebras.
We compute $\Phi_n(P)$ for several families of symmetric polynomials $P$:
symplectic and orthogonal Schur polynomials,
elementary symmetric polynomials,
complete homogeneous polynomials, and power sums.
Some of these formulas were already found by
Elouafi (2014) and Lachaud (2016).

The polynomials of the form
$\Phi_n(\operatorname{s}_{\lambda/\mu}^{(2n)})$,
where $\operatorname{s}_{\lambda/\mu}^{(2n)}$
is a skew Schur polynomial in $2n$ variables,
arise naturally in the study of the minors
of symmetric banded Toeplitz matrices,
when the generating symbol is a palindromic Laurent polynomial,
and its roots can be written as
$x_1,\ldots,x_n,x^{-1}_1,\ldots,x^{-1}_n$.
Trench (1987) and Elouafi (2014) found efficient formulas
for the determinants of symmetric banded Toeplitz matrices.
We show that these formulas are equivalent
to the result of Ciucu and Krattenthaler (2009)
about the factorization of the characters of classical groups.
\end{abstract}




\clearpage
\tableofcontents

\section{Introduction and main results}

In this paper we study symmetric polynomials $P$ in $2n$ variables
evaluated at the symplectic alphabet:
\begin{equation}\label{eq:P_at_symplectic_alphabet}
P(x_1,\ldots,x_n,x_1^{-1},\ldots,x_n^{-1}).
\end{equation}
We show how to rewrite such expressions
in terms of the ``Dickson--Zhukovsky variables'' $z_j \coloneqq x_j+x_j^{-1}$.
The function $t\mapsto t+t^{-1}$
is widely known as the Zhukovsky transform;
some authors~\cite{GemignaniNoferini2013}
relate it with the name of Dickson.
The symplectic alphabet $x_1,\ldots,x_n,x_1^{-1},\ldots,x_n^{-1}$
naturally arises as the list of the roots
of palindromic (i.e.\ self-reciprocal) polynomials,
see more details in Section~\ref{sec:palindromic_polynomials}.
In particular, expressions of the form~\eqref{eq:P_at_symplectic_alphabet}
appear in the following situations.

\begin{enumerate}
\item If $A$ is a symplectic matrix
or a special orthogonal matrix of even order,
then the characteristic polynomial of $A$ is palindromic.
If $P$ is a symmetric polynomial in $2n$ variables,
then $P$ evaluated at the eigenvalues of $A$
is an expression of the form~\eqref{eq:P_at_symplectic_alphabet}.
\item The characters of symplectic groups
or special orthogonal groups of even orders
are particular cases of~\eqref{eq:P_at_symplectic_alphabet}.
See more details in Section~\ref{sec:Schur_Chebyshev_quotients}.
\item Given a palindromic Laurent polynomial $a$,
consider the banded symmetric Toeplitz matrices $T_m(a)$ generated by $a$.
The minors of $T_m(a)$, expressed in terms of the roots of $a$,
are of the form~\eqref{eq:P_at_symplectic_alphabet}.
See more details in Section~\ref{sec:symm_Toeplitz_minors}.
\end{enumerate} 

This paper is inspired by Elouafi's article \cite{Elouafi2014}
on the determinants of banded symmetric Toeplitz matrices.
In the process of preparation of the paper,
we found a paper by Lachaud~\cite{Lachaud2016}
which contains ``our'' Proposition~\ref{prop:elem_sympl_via_elem_z}.

We work over the field $\bC$,
though some results can be extended
to other fields of characteristic $0$.
Let $n$ be a fixed natural number.
We denote by $\Sym_n$ the algebra
of symmetric polynomials in $n$ variables.

\begin{thm}\label{thm:general}
Let $P\in\Sym_{2n}$.
Then there exists a unique $Q$ in $\Sym_n$ such that
\[
P\left(x_1,\ldots,x_n,x_1^{-1},\ldots,x_n^{-1}\right)
=Q\left(x_1+x_1^{-1},\ldots,x_n+x_n^{-1}\right).
\]
\end{thm}

An analog (and also a corollary) of Theorem~\ref{thm:general}
for the odd symplectic alphabet
\[
x_1,\ldots,x_n,x_1^{-1},\ldots,x_n^{-1},1,
\] 
is stated below.

\begin{thm}\label{thm:general_odd}
Let $P\in\Sym_{2n+1}$.
Then there exists a unique $Q$ in $\Sym_n$ such that
\[
P\left(x_1,\ldots,x_n,x_1^{-1},\ldots,x_n^{-1},1\right)
=Q\left(x_1+x_1^{-1},\ldots,x_n+x_n^{-1}\right).
\]
\end{thm}

Chebyshev polynomials of the first, second, third, and fourth kind,
denoted by $\cT_m$, $\cU_m$, $\cV_m$, $\cW_m$, respectively,
play an important role in this paper.
In particular, the polynomials $2\cT_m(t/2)$ and $\cU_m(t/2)$
convert the Dickson--Zhukovsky variable $z_j \coloneqq x_j+x_j^{-1}$
into the power sum and the complete homogeneous polynomial
in $x_j$ and $x_j^{-1}$:
\[
2\cT_m(z_j/2)=x_j^m+x_j^{-m}=\powersum_m(x_j,x_j^{-1}),\qquad
\cU_m(z_j/2)=\sum_{k=0}^m x_j^{m-2k}=\hom_m(x_j,x_j^{-1}).
\]
Section~\ref{sec:Chebyshev} lists
necessary properties of Chebyshev polynomials,
Section~\ref{sec:palindromic_polynomials} considers
palindromic univariate polynomials and their roots,
and Section~\ref{sec:proof1} contains proofs of
Theorem~\ref{thm:general} and~\ref{thm:general_odd}.

In the situations of Theorems~\ref{thm:general}~and~\ref{thm:general_odd}, we denote $Q$ by $\Phi_n(P)$ and $\Phiodd_n(P)$, respectively.
Clearly, the functions $\Phi_n\colon\Sym_{2n}\to\Sym_n$
and $\Phiodd_n\colon\Sym_{2n+1}\to\Sym_n$,
defined by these rules,
are linear and multiplicative,
i.e.\ $\Phi_n$ and $\Phiodd_n$ are homomorphisms of algebras.
Example~\ref{example:not_injective} shows that
$\Phi_n$ and $\Phi_n^{\odd}$ are not injective.

In this paper we freely apply
some well-known properties of symmetric polynomials,
see~\cite{Macdonald1995}
or~\cite[Appendix~A]{FultonHarris1991} as reference.
Let $\Partitions$ be the set of all integer partitions.
Given a partition $\la$, $\ell(\la)$ and $|\la|$
denote the length and the weight of $\la$, respectively.
Let $\Partitions_n$ be the set of all integer partitions $\la$
with $\ell(\la)\le n$.
We denote by $\Sym$ the algebra of symmetric functions
and use the following bases of $\Sym$.
\begin{center}
\begin{tabular}{cl} 
\toprule
Symbol & Family \\
\midrule
$\elem_\la$ & elementary functions \\
$\hom_\la$ & complete homogeneous functions \\
$\powersum_\la$ & power sum functions
\\
$\schur_\la$ & Schur functions \\
$\sympl_\la$ & symplectic Schur functions \\
$\orth_\la$ & orthogonal Schur functions \\
\bottomrule
\end{tabular}
\end{center}
The functions $\sympl_\la$ and $\orth_\la$
are defined by Jacobi--Trudi formulas,
see Section~\ref{sec:Schur_Chebyshev_quotients}.

\bigskip

For the symmetric functions introduced above, we write the super-index ${(n)}$ to indicate their restrictions
to $n$ variables, i.e. the corresponding elements of $\Sym_n$.
For example, $\schur_\la^{(n)}$ is the Schur polynomial
in $n$ variables associated to the partition $\la$.
Furthermore, put
\[
\elemz^{(n)}_\la\eqdef\Phi_n(\elem^{(2n)}_\la),\quad
\schurz^{(n)}_\la\eqdef\Phi_n(\schur^{(2n)}_\la),\quad
\spz^{(n)}_\la\eqdef\Phi_n(\sympl^{(2n)}_\la),\quad
\oz^{(n)}_\la\eqdef\Phi_n(\orth^{(2n)}_\la),
\]
etc.
In other words, $\schurz^{(n)}_\la$
is obtained from $\schur^{(2n)}_\la$
by applying Theorem~\ref{thm:general}
and passing from the ``symplectic alphabet''
$x_1,\ldots,x_n,x^{-1}_1,\ldots,x^{-1}_n$
to the ``Dickson--Zhukovsky variables'' $z_j=x_j+x^{-1}_j$:
\[
\schurz^{(n)}_\la(z_1,\ldots,z_n)
\eqdef\schur^{(2n)}_\la(x_1,\ldots,x_n,x_1^{-1},\ldots,x_n^{-1}).
\]
Similarly, put
\[
\schurz^{\odd,(n)}_\la\eqdef\Phiodd_n(\schur^{(2n+1)}_\la),\quad
\spz^{\odd,(n)}_\la\eqdef\Phiodd_n(\sympl^{(2n+1)}_\la),\quad
\oz^{\odd,(n)}_\la\eqdef\Phiodd_n(\orth^{(2n+1)}_\la),
\]
etc.
For the sake of brevity, we will omit the superindex $(n)$,
when indicating the list of variables $z=(z_1,\ldots,z_n)$.

The next theorem, proven in Section~\ref{sec:Schur_Chebyshev_quotients}, yields convenient bialternant formulas
for $\spz_\la$, $\spz^{\odd}_\la$, $\oz_\la$, and $\oz^{\odd}_\la$,
representing them as ``Schur--Chebyshev quotients''.
We denote by $\cTmonic_m$ the monic integer version
of the Chebyshev $\cT_m$ polynomial,
i.e.\ $\cTmonic_m(u)\eqdef 2\cT_m(u/2)$ for $m>0$
and $\cTmonic_0(u)\eqdef 1$.
The polynomial $\cU_m^{(1)}$ is defined as $\sum_{k=0}^m\cU_k$.

\begin{thm}\label{thm:spz_oz_bialternant}
For every $\la$ in $\Partitions_n$,
\begin{align}
\label{eq:spz_as_uquotient}
\spz_\la(z_1,\ldots,z_n)
&=
\frac{\det\bigl[\cU_{\la_j+n-j}(z_k/2)\bigr]_{j,k=1}^n}%
{\det\bigl[\cU_{n-j}(z_k/2)\bigr]_{j,k=1}^n}
=
\frac{\det\bigl[\cU_{\la_j+n-j}(z_k/2)\bigr]_{j,k=1}^n}%
{\prod_{1\le j<k\le n}(z_j-z_k)},
\\
\label{eq:oz_as_tquotient}
\oz_\la(z_1,\ldots,z_n)
&=
\frac{\det\bigl[\cTmonic_{\la_j+n-j}(z_k)\bigr]_{j,k=1}^n}%
{\det\bigl[\cTmonic_{n-j}(z_k/2)\bigr]_{j,k=1}^n}
=
\frac{\det\bigl[\cTmonic_{\la_j+n-j}(z_k)\bigr]_{j,k=1}^n}%
{\prod_{1\le j<k\le n}(z_j-z_k)},
\\
\label{eq:spzodd_as_usumquotient}
\spz^{\odd}_\la(z_1,\ldots,z_n)
&=
\frac{\det\bigl[\cU^{(1)}_{\la_j+n-j}(z_k/2)\bigr]_{j,k=1}^n}%
{\det\bigl[\cU^{(1)}_{n-j}(z_k/2)\bigr]_{j,k=1}^n}
=
\frac{\det\bigl[\cU^{(1)}_{\la_j+n-j}(z_k/2)\bigr]_{j,k=1}^n}%
{\prod_{1\le j<k\le n}(z_j-z_k)},
\\
\label{eq:ozodd_as_wquotient}
\oz^{\odd}_\la(z_1,\ldots,z_n)
&=
\frac{\det\bigl[\cW_{\la_j+n-j}(z_k/2)\bigr]_{j,k=1}^n}%
{\det\bigl[\cW_{n-j}(z_k/2)\bigr]_{j,k=1}^n}
=
\frac{\det\bigl[\cW_{\la_j+n-j}(z_k/2)\bigr]_{j,k=1}^n}%
{\prod_{1\le j<k\le n}(z_j-z_k)},
\\
\label{eq:ozodd_neg_as_vquotient}
(-1)^{|\la|}\oz^{\odd}_\la(-z_1,\ldots,-z_n)
&=\frac{\det\bigl[\cV_{\la_j+n-j}(z_k/2)\bigr]_{j,k=1}^n}%
{\det\bigl[\cV_{n-j}(z_k/2)\bigr]_{j,k=1}^n}
=\frac{\det\bigl[\cV_{\la_j+n-j}(z_k/2)\bigr]_{j,k=1}^n}%
{\prod_{1\le j<k\le n}(z_j-z_k)}.
\end{align}
\end{thm}

Our main results, stated below
as Theorems~\ref{thm:pal_elem},
\ref{thm:pal_hom}, \ref{thm:pal_powersum},
\ref{thm:minimal_generating_subsets}
and proven in Section~\ref{sec:properties_pal_e_h_p},
are properties of $\homz_m^{(n)}$,
$\elemz_m^{(n)}$, and $\powersumz_m^{(n)}$.
Formula~\eqref{eq:elemz_via_elem} was found by
Lachaud~\cite[Theorem~A.2 and proof of Lemma~A.3]{Lachaud2016},
and formula~\eqref{eq:palhom_as_sum_of_U}
is similar to one part of Elouafi~\cite[Lemma~3]{Elouafi2014}.

\begin{thm}\label{thm:pal_elem}
For $m$ in $\{0,\ldots,2n\}$,
\begin{align}
\label{eq:elemz_via_elem}
\elemz_m(z_1,\ldots,z_n)
 &= \sum_{k=\max\{0,m-n\}}^{\lfloor m/2\rfloor}
\binom{n - m + 2k}{k} \elem_{m-2k}(z_1,\ldots,z_n).
\intertext{For $m$ in $\{0,\ldots,n\}$,}
\label{eq:epal_as_tquotient}
\elemz_m(z_1,\ldots,z_n)
&=\oz_{(1^m)}(z_1,\ldots,z_n),
\\
\label{eq:elemz_as_det}
\elemz_m(z_1,\ldots,z_n)
&=
\frac{1}{\prod_{1\le j<k\le n}(z_j-z_k)}
\detmatr{ccc}{
\cTmonic_n(z_1) & \ldots & \cTmonic_n(z_n) \\
\vdots & \ddots & \vdots \\
\cTmonic_{n-m+1}(z_1) & \ldots & \cTmonic_{n-m+1}(z_n) \\[1ex]
\cTmonic_{n-m-1}(z_1) & \ldots & \cTmonic_{n-m-1}(z_n) \\
\vdots & \ddots & \vdots \\
\cTmonic_0(z_1) & \ldots & \cTmonic_0(z_n)
},
\\
\label{eq:elemz_via_uquotients}
\elemz_m(z_1,\ldots,z_n)
&=\sum_{k=0}^{\lfloor m/2\rfloor}
\spz_{(1^{m-2k})}(z_1,\ldots,z_n),
\\
\label{eq:elem_via_elemz}
\elem_m(z_1,\ldots,z_n)
&=\sum_{k=0}^{\lfloor m/2\rfloor}
(-1)^k \coefT_{n-m+2k,k}
\elemz_{m-2k}(z_1,\ldots,z_n), 
\end{align}
where
\begin{equation}\label{eq:coefT}
\coefT_{s,k}\eqdef
\begin{cases}
\displaystyle \frac{(s-k-1)!\,s}{k!\,(s-2k)!}, &
s\in\bN,\ 0\le k\le s;\\[1ex]
1, & k=s=0.
\end{cases}
\end{equation}
\end{thm}

In the notation for partitions
(in particular, in formula~\eqref{eq:epal_as_tquotient}),
$p^q$ is the number $p$ repeated $q$ times.
For example, $(3^2,0^4)$ means $(3,3,0,0,0,0)$.

\begin{thm}\label{thm:pal_hom}
Let $m\in\bN_0$. Then
\begin{align}
\label{eq:palhom_as_sum_of_U}
\homz_m(z_1,\ldots,z_n)
&= \sum_{j=1}^n \frac{\cU_{m+n-1}(z_j/2)}{\displaystyle \prod_{k\in \lbrace1,\ldots,n\rbrace \setminus \lbrace j \rbrace }(z_j-z_k)},
\\
\label{eq:homz_as_det}
\homz_m(z_1,\ldots,z_n)
&=
\frac{1}{\prod_{1\le j<k\le n}(z_j-z_k)}
\detmatr{ccc}{
\cU_{m+n-1}(z_1/2) & \ldots & \cU_{m+n-1}(z_n/2) \\
\cU_{n-2}(z_1/2) & \ldots & \cU_{n-2}(z_n/2) \\
\vdots & \ddots & \vdots \\
\cU_0(z_1/2) & \ldots & \cU_0(z_n/2)
},
\\
\label{eq:palhom_via_hom}
\homz_m(z_1,\ldots,z_n)
&= \sum_{k=0}^{\lfloor m/2\rfloor}
(-1)^k\binom{n + m - k - 1}{k}\hom_{m-2k}(z_1,\ldots,z_n),
\\
\label{eq:palhom_as_sum_of_prod_U}
\homz_m(z_1,\ldots,z_n)
&=
\sum_{\substack{\al_1,\ldots,\al_n\in\bN_0 \\ \al_1+\cdots+\al_n=m}}
\prod_{j=1}^n \cU_{\al_j}\left(z_j/2\right),
\\
\label{eq:palhom_as_spz}
\homz_m(z_1,\ldots,z_n)
&=\spz_{(m)}(z_1,\ldots,z_n),
\\
\label{eq:palhom_via_oz}
\homz_m(z_1,\ldots,z_n)
&=\sum_{k=0}^{\lfloor m/2\rfloor}\oz_{(m-2k)}(z_1,\ldots,z_n),
\\
\label{eq:hom_via_palhom}
\hom_m(z_1,\ldots,z_n)
&=\sum_{k=0}^{\lfloor m/2\rfloor}
\frac{(m+n-1)!\,(m+n-2k)}{k!\,(m+n-k)!}
\homz_{m-2k}(z_1,\ldots,z_n).
\end{align}
\end{thm}

\begin{thm}\label{thm:pal_powersum}
Let $m\in\bN_0$. Then
\begin{align}
\label{eq:pal_powersum_via_ChebT_z}
\powersumz_m(z_1,\ldots,z_n)
&=\sum_{j=1}^n 2\cT_m(z_j/2),\\
\label{eq:pal_powersum_as_sum_of_powersum}
\powersumz_m(z_1,\ldots, z_n)
&=
\begin{cases}
\displaystyle m\sum_{j=0}^{\lfloor m/2 \rfloor}
\frac{(-1)^j}{m-j}\binom{m-j}{j}\powersum_{m-2j}(z_1,\ldots,z_n), & m\in\bN;\\[0.5ex]
2\powersum_0(z_1,\ldots,z_n), & m=0,
\end{cases}
\\
\label{eq:powersum_as_sum_palpowersum}
\powersum_m(z_1,\ldots ,z_n)
&=\sum_{k=0}^{\lfloor m/2\rfloor}
\coefmonomialviaT_{m,k}\powersumz_{m-2k}(z_1,\ldots ,z_n),
\end{align}
where the coefficients $\coefmonomialviaT_{m,k}$ are defined by
\begin{equation}\label{eq:coef_mon_as_sum_T}
\coefmonomialviaT_{m,k}
\coloneqq
\begin{cases}
\displaystyle\binom{m}{k}, &
\text{if}\ k<\frac{m}{2};
\\[2ex]
\displaystyle\frac{1}{2}\binom{m}{m/2}, &
\text{if}\ k=\frac{m}{2}.
\end{cases}
\end{equation}
\end{thm}

Notice that the coefficients in
\eqref{eq:elemz_via_elem}, \eqref{eq:elem_via_elemz},
\eqref{eq:palhom_via_hom}, and \eqref{eq:hom_via_palhom}
depend on $n$.
For example,
\[
\homz_2(z_1,\ldots,z_n)=\hom_2(z_1,\ldots,z_n)-n.
\]
Therefore, $\homz^{(n)}_2$ is well-defined only as an element of $\Sym_n$,
i.e. there is no function in $\Sym$ that could be denoted by $\homz_2$.
In contrast to this, the coefficients in
\eqref{eq:pal_powersum_as_sum_of_powersum}
and \eqref{eq:powersum_as_sum_palpowersum}
do not depend on $n$,
but in these formulas $\powersum_0^{(n)} \coloneqq n$
and $\powersumz_0^{(n)} \coloneqq  2n$.

\begin{tempremark}
We are going to modify or delete this remark
in the final version of the paper.
We can avoid $\powersum_0^{(n)}$ and $\powersumz_0^{(n)}$,
i.e. replace them by the explicit constants $n$ and $2n$, respectively.
Then in some formulas we will need
to separate the corresponding summands
and to multiply them by
the conditional factor $[m\ \text{is even}]$.
So, some formulas are nicer-looking
with $\powersum_0^{(n)}$ and $\powersumz_0^{(n)}$.
\end{tempremark}

\begin{thm}\label{thm:minimal_generating_subsets}
Each of the sets
$\{\elemz^{(n)}_m\}_{m=1}^n$,
$\{\homz^{(n)}_m\}_{m=1}^n$,
$\{\powersumz^{(n)}_m\}_{m=1}^n$
is a minimal generating subset of the unital algebra $\Sym_n$,
and the homomorphism $\Phi_n$ is surjective.
\end{thm}

As a consequence of
Theorems~\ref{thm:general} and~\ref{thm:minimal_generating_subsets},
the algebra $\Sym_n$ is isomorphic
to the algebra of the expressions of the form
$P(x_1,\ldots,x_n,x_1^{-1},\ldots,x_n^{-1})$, where $P\in\Sym_{2n}$,
via the isomorphism $Q\mapsto Q(x+x^{-1})$.
This fact is close to some ideas from~\cite[Appendix~A]{Lachaud2016}.

\begin{thm}\label{thm:basis}
Each of the families
$(\schurz^{(n)}_\la)_{\la\in\Partitions_n}$,
$(\spz^{(n)}_\la)_{\la\in\Partitions_n}$,
$(\oz^{(n)}_\la)_{\la\in\Partitions_n}$,
$(\schurz^{\odd,(n)}_\la)_{\la\in\Partitions_n}$,
$(\spz^{\odd,(n)}_\la)_{\la\in\Partitions_n}$,
$(\oz^{\odd,(n)}_\la)_{\la\in\Partitions_n}$
is a basis of the vector space $\Sym_n$.
\end{thm}

Next we state analogs of the Cauchy and dual Cauchy identities.

\begin{thm}[Cauchy identities]\label{thm:Cauchy_palindromic}
For $z=(z_1,\ldots,z_n)$ and $y=(y_1,\ldots,y_m)$,
\begin{align}
\label{eq:Cauchy_palindromic}
\sum_{\la\in\Partitions}
\schurz_\la(z)\schur_\la(y)
&=\prod_{j=1}^n\prod_{k=1}^m\frac{1}{1-z_j y_k+y_k^2},
\\
\label{eq:dual_Cauchy_palindromic}
\sum_{\la\in\Partitions}
\schurz_\la(z)\schur_{\la'}(y)
&=\prod_{j=1}^n\prod_{k=1}^m\left(1+z_j y_k+y_k^2\right),
\\
\label{eq:dual_Cauchy_schurz_schurz}
\sum_{\la\in\Partitions}
\schurz_\la(z)\schurz_{\la'}(y)
&=\prod_{j=1}^n\prod_{k=1}^m (z_j+y_k)^2,
\\
\label{eq:Cauchy_palindromic_odd}
\sum_{\la\in\Partitions}
\schurz^{\odd}_\la(z)
\schur_\la(y)
&=\left(\prod_{j=1}^n\prod_{k=1}^m\frac{1}{1-z_jy_k+y_k^2}\right)\,
\left(\prod_{k=1}^m \frac{1}{1-y_k}\right),
\\
\label{eq:dual_Cauchy_palindromic_odd}
\sum_{\la\in\Partitions}
\schurz^{\odd}_\la(z)\schur_{\la'}(y)
&=
\left(\prod_{j=1}^n\prod_{k=1}^m (1+z_j y_k+y_k^2) \right)
\left(\prod_{k=1}^m (1+y_k)\right),
\\
\label{eq:dual_Cauchy_schurzodd_schurzodd}
\sum_{\la\in\Partitions}
\schurz^{\odd}_\la(z)\schurz^{\odd}_{\la'}(y)
&=2\left(\prod_{j=1}^n \prod_{k=1}^m (z_j+y_k)^2\right)
\left(\prod_{j=1}^n (2+z_j)\right)
\left(\prod_{k=1}^m (2+y_k)\right).
\end{align}
\end{thm}

There are also analogs of Cauchy identities with
$\spz_\la(z)$ or $\oz_\la(z)$ instead of $\schurz_\la(z)$;
we omit them for the sake of brevity.

Section~\ref{sec:palschur} contains
proofs of Theorems~\ref{thm:basis} and \ref{thm:Cauchy_palindromic},
and some other simple facts about $\schurz_\la$.
A natural problem is to find a bialternant formula for $\schurz_\la$,
similar to the formulas from Theorem~\ref{thm:spz_oz_bialternant}.
In Proposition~\ref{prop:no_bialternant_schurz}
we prove that there is no bialternant formula for
$\schurz_{(2,1)}(z_1,z_2)$, with denominator $z_1-z_2$.
In Example~\ref{example:bialternant_schurz_21}
we give a bialternant formula for $\schurz_{(2,1)}(z_1,z_2)$
with a more complicated denominator.

In Section~\ref{sec:symm_Toeplitz_minors} we show
that there is a surjective (but non-injective) correspondence
between minors of banded symmetric Toeplitz matrices
and polynomials $\schurz_{\la/\mu}$,
obtained by applying $\Phi_n$ to skew Schur polynomials.
In particular, the banded symmetric Toeplitz determinants
correspond to $\schurz_{(m^n)}$,
and the factorization of
$\schur_{(m^n)}(x_1,\ldots,x_n,x_1^{-1},\ldots,x_n^{-1})$
proven in 2009 by~Ciucu and Krattenthaler~\cite{CiucuKrattenthaler2009}
is equivalent to the formulas
for banded symmetric Toeplitz determinants,
found independently by Trench~\cite{Trench1987symm} in 1987
and by Elouafi~\cite{Elouafi2014} in 2014.
Elouafi's formula has been used in investigations
about the eigenvalues of symmetric Toeplitz matrices~\cite{BG2017}.

All formulas from Theorems~\ref{thm:spz_oz_bialternant},
\ref{thm:pal_elem}, \ref{thm:pal_hom}, \ref{thm:pal_powersum},
and \ref{thm:Cauchy_palindromic}
are thoroughly tested in Sagemath~\cite{SageMath}
for small values of parameters.
We share the corresponding Sagemath code at the page
\myurl{http://www.egormaximenko.com/programs/tests\_palindromic.html}.

\section{Necessary facts about Chebyshev polynomials}
\label{sec:Chebyshev}
 
Most of the material in this section can be found
in~\cite{MasonHandscomb2003Chebyshev}.
Four families of Chebyshev polynomials
can be defined by the same recurrent formula
\[
\cT_m(u)=2u\,\cT_{m-1}(u)-\cT_{m-2}(u),\qquad
\cU_m(u)=2u\,\cU_{m-1}(u)-\cU_{m-2}(u),\qquad \ldots,
\]
with the initial conditions $\cT_0(u)=\cU_0(u)=\cV_0(u)=\cW_0(u)=1$,
\[
\cT_1(u)=u,\quad \cU_1(u)=2u,\quad
\cV_1(u)=2u-1,\quad \cW_1(u)=2u+1.
\]
The Chebyshev polynomials have the following important properties:
\begin{align}
\label{eq:ChebT_main_property}
2\cT_m\left(\frac{1}{2}(t+t^{-1})\right)
&=t^m+t^{-m},
\\
\label{eq:ChebU_main_property}
\cU_m\left(\frac{1}{2}(t+t^{-1})\right)
&=\frac{t^{m+1}-t^{-m-1}}{t-t^{-1}},
\\
\label{eq:ChebV_main_property}
\cV_m\left(\frac{1}{2}(t^2+t^{-2})\right)
&=\frac{t^{2m+1}+t^{-2m-1}}{t+t^{-1}},
\\
\label{eq:ChebW_main_property}
\cW_m\left(\frac{1}{2}(t^2+t^{-2})\right)
&=\frac{t^{2m+1}-t^{-2m-1}}{t-t^{-1}}.
\end{align}
The generating functions of the sequences
$(2\cT_m(u/2))_{m=0}^\infty$ and $(\cU_m(u/2))_{m=0}^\infty$
are given by
\begin{align}
\label{eq:T_generating}
\frac{2-tu}{1-tu+t^2}&=\sum_{m=0}^\infty 2\cT_m(u/2)t^m,
\\
\label{eq:U_generating}
\frac{1}{1-tu+t^2}&=\sum_{m=0}^\infty \cU_m(u/2)t^m.
\end{align}
Using \eqref{eq:T_generating} and \eqref{eq:U_generating},
it is easy to derive explicit formulas for $\cT_m$ and $\cU_m$:
\begin{align}
\label{eq:T_explicit}
2\cT_m(u/2)
&=m\sum_{k=0}^{\lfloor m/2\rfloor}
\frac{(-1)^k}{m-k}\binom{m-k}{k}u^{m-2k}\qquad(m\in\bN),
\\
\label{eq:U_explicit}
\cU_m(u/2)
&=\sum_{k=0}^{\lfloor m/2\rfloor}
(-1)^k \binom{m-k}{k}u^{m-2k}\qquad(m\in\bN_0).
\end{align}
With the notation $\coefT_{m,k}$ defined by~\eqref{eq:coefT},
we rewrite~\eqref{eq:T_explicit} in the form
\begin{equation}\label{eq:T_explicit_via_coefT}
\sum_{k=0}^{\lfloor m/2\rfloor}
(-1)^k \coefT_{m,k}u^{m-2k}
=\cTmonic_m(u)=
\begin{cases}
2\cT_m(u/2), & m\in\bN; \\
\cT_0(u/2), & m=0.
\end{cases}
\end{equation}
The monomials $u^m$ ($m\in\bN_0$) can be written
as linear combinations of $\cT_{m-2k}(u/2)$ or $\cU_{m-2k}(u/2)$:
\begin{align}
\label{eq:monomial_as_lin_comb_T}
u^m&=\sum_{k=0}^{\lfloor m/2 \rfloor}
\coefmonomialviaT_{m,k}2\cT_{m-2k}(u/2),
\\
\label{eq:monomial_as_lin_comb_U}
u^m&=\sum_{k=0}^{\lfloor m/2\rfloor}
\Catalan_{m-k,k}\cU_{m-2k}(u/2),
\end{align}
where $\coefmonomialviaT_{m,k}$ is defined by~\eqref{eq:coef_mon_as_sum_T},
and $\Catalan_{m,k}$ are the elements of Catalan's triangle:
\[
\Catalan_{m,k}\eqdef\frac{(m+k)!\,(m-k+1)}{k!\,(m+1)!}.
\]
The forthcoming ``duplication formulas''
follow easily from
\eqref{eq:ChebT_main_property}--\eqref{eq:ChebW_main_property}
and can be found in
\cite[Section 1.2.4]{MasonHandscomb2003Chebyshev}.

\begin{prop}
For every $m$ in $\bN_0$,
\begin{align}
\label{eq:duplication_T}
\cT_{2m}(t/2)&=\cT_m((t^2-2)/2),
\\
\label{eq:duplication_U}
\cU_{2m+1}(t/2)&=t\,\cU_m((t^2-2)/2),
\\
\label{eq:duplication_V}
2\cT_{2m+1}(t/2)&=t\,\cV_m((t^2-2)/2),
\\
\label{eq:duplication_W}
\cU_{2m}(t/2)&=\cW_m((t^2-2)/2).
\end{align}
\end{prop}

The polynomials $\cV_m$ and $\cW_m$ are related by
\begin{equation}\label{eq:V_and_W}
\cV_m(-t)=(-1)^m\cW_m(t).
\end{equation}
We denote by $\cU_m^{(1)}$ the polynomial $\sum_{k=0}^m\cU_k$
and by $\coefsumU_{m,j}$ the coefficient of $t^j$ in $\cU_m^{(1)}(t/2)$:
\begin{equation}\label{eq:cU_sum}
\cU_m^{(1)}(t/2)
=\sum_{k=0}^m \cU_k(t/2)
=\sum_{j=0}^m \coefsumU_{m,j} t^j.
\end{equation}
It follows from~\eqref{eq:U_explicit} that
\begin{equation}\label{eq:coefsumU}
\coefsumU_{m,k}
=\sum_{j=k}^{\lfloor(m+k)/2\rfloor} (-1)^{j-k} \binom{j}{k}
=\sum_{j=0}^{\lfloor(m-k)/2\rfloor} (-1)^{j} \binom{k+j}{j}.
\end{equation}
The coefficients~$\coefsumU_{m,k}$ form the sequence A128494
in the Online Encyclopedia of Integer Sequences \cite{OEIS_A128494}.
It is easy to prove by induction that
\begin{equation}\label{eq:sumU_via_V}
\cU^{(1)}_m(t/2)=\frac{\cV_{m+1}(t/2)-1}{t-2}.
\end{equation}

\section{Palindromic univariate polynomials and their roots}
\label{sec:palindromic_polynomials}

A univariate polynomial $f(t)=\sum_{k=0}^m a_k t^k$, with $a_m\ne0$,
is called \emph{palindromic} if $a_k=a_{m-k}$ for all $k$
in $\{0,\ldots,m\}$.
This condition is equivalent to the identity $f(t)=t^m f(1/t)$.
In this section we review some known facts
about palindromic polynomials of even degrees,
then make a couple of remarks about palindromic polynomials of odd degrees.
For $m=2n$, the main observation is that
$f(t)$ can be written as $t^n g(t+t^{-1})$,
where $g$ is a certain polynomial.
This idea appears, for example,
in Dickson~\cite[Chapter VIII]{Dickson1914},
but without explicit formula for $g$.
The explicit formula for $g$,
i.e.\ Proposition~\ref{prop:palindromic_pol_through_ChebT} below,
was discovered independently by many authors.
See, for example, Elouafi~\cite[formula (2.2)]{Elouafi2014},
Lachaud~\cite[Remark~A.4]{Lachaud2016}.

\begin{tempremark}
Wikipedia (``Reciprocal polynomial'')
mentions that the formula $f(t)=t^n g(t+t^{-1})$
was published by Durand in
``Solutions num\'{e}riques des \'{e}quations algr\'{e}briques I'' (1961),
but we are unable to find that text.
So, unfortunately, we don't know who was first
to discover the explicit formula for $g$ in terms of $\cT_m$.
\end{tempremark}

\begin{prop}\label{prop:palindromic_pol_defined_by_roots}
Let $x_1,\ldots,x_n\in\bC$.
For each $j$ in $\{1,\ldots,n\}$, put $z_j=x_j+x_j^{-1}$.
Define univariate polynomials $f$ and $g$ by
\[
f(t)\eqdef\prod_{j=1}^n \bigl((t-x_j)(t-x_j^{-1})\bigr),\qquad
g(u)\eqdef\prod_{j=1}^n (u-z_j).
\]
Then
\begin{equation}\label{eq:f_via_g1}
f(t)=t^n g\left(t+t^{-1}\right),
\end{equation}
and the polynomial $f$ is palindromic.
\end{prop}

\begin{proof}
Formula~\eqref{eq:f_via_g1} follows directly from the definition of $f$ and $g$:
\[
t^n g(t+t^{-1})
=t^n\prod_{j=1}^n \bigl(t+t^{-1}-x_j-x_j^{-1}\bigr)
=\prod_{j=1}^n \bigl((t-x_j)(t-x_j^{-1})\bigr)
=f(t).
\]
The expression $g(t+t^{-1})$, being a linear combination of expressions
of the form $(t+t^{-1})^k$, is a palindromic Laurent polynomial
of the form $c_0+\sum_{k=1}^n c_k (t^k+t^{-k})$.
Now~\eqref{eq:f_via_g1} implies that $f$ is a palindromic polynomial.
\end{proof}

The next propositions are, in a certain sense,
inverse to Proposition~\ref{prop:palindromic_pol_defined_by_roots}.
Now we define $g$ through the coefficients of $f$
and make conclusions about the roots of $f$ and $g$.

\begin{prop}\label{prop:palindromic_pol_through_ChebT}
Let $f$ be a palindromic univariate polynomial with complex coefficients:
\[
f(t)=\sum_{k=0}^{2n} a_k t^k,
\]
where $a_{2n-k}=a_k$ for every $k$ in $\{0,\ldots,n\}$.
Define a univariate polynomial $g$ by
\[
g(u)\eqdef
\sum_{j=0}^n a_{n-j}\cTmonic_j(u)
=a_n+\sum_{j=1}^n 2a_{n-j}\cT_j(u/2).
\]
Then
\begin{equation}\label{eq:f_via_g}
f(t)=t^n g\left(t+t^{-1}\right).
\end{equation}
\end{prop}

\begin{proof}
For every $k$  in $\{0,\ldots,n-1\}$,
write $a_k t^k+a_{2n-k} t^{2n-k}$
as $a_k t^n (t^{n-k}+t^{k-n})$,
then apply~\eqref{eq:ChebT_main_property}.
\end{proof}

\begin{prop}\label{prop:roots_of_palindromic_pol}
Let $f$ and $g$ be as in Proposition~\ref{prop:palindromic_pol_through_ChebT},
$a_0\ne0$, and $z_1,\ldots,z_n$ be
the roots of the polynomial $g$:
\[
g(u)=a_0\prod_{j=1}^n (u-z_j).
\]
For each $j$, denote by $x_j$ a complex number satisfying
$x_j+x^{-1}_j=z_j$.
Then the numbers
\[
x_1,\ldots,x_n,x_1^{-1},\ldots,x_n^{-1}
\]
are the roots of the polynomial $f$, i.e.
\[
f(t)=a_0\prod_{j=1}^n
\left(\left(t-x_j\right)\left(t-x_j^{-1}\right)\right).
\]
\end{prop}

\begin{proof}
Follows directly from~\eqref{eq:f_via_g}
and the definitions of $z_j$ and $x_j$:
\[
f(t)
=t^n g\left(t+t^{-1}\right)
=a_0 t^n \prod_{j=1}^n
\left(t+t^{-1}-x_j-x_j^{-1}\right)
=a_0\prod_{j=1}^n
\left(\left(t-x_j\right)\left(t-x_j^{-1}\right)\right).
\qedhere
\]
\end{proof}

Here is a simple result about palindromic polynomials of odd degrees.

\begin{prop}\label{prop:palindromic_polynomial_of_odd_degree}
Let $f$ be a palindromic polynomial of degree $2n+1$.
Then there exists a unique palindromic polynomial $g$ of degree $2n$
such that $f(t)=(t+1)g(t)$.
The zeros of $f$ can be written as
\[
x_1,\ldots,x_n,x_1^{-1},\ldots,x_n^{-1},-1.
\]
\end{prop}

If $P\in\Sym_{2n+1}$ and $P$ is homogeneous of degree $d$, then
\[
P(x,x^{-1},-1)=(-1)^d P(-x,-x^{-1},1).
\]
In this situation, it is possible to work over the alphabet
 $x,x^{-1},1$, instead of $x,x^{-1},-1$.

\section{\texorpdfstring{Construction of the morphisms
$\boldsymbol{\Phi_n}$ and $\boldsymbol{\Phiodd_n}$}%
{Construction of the morphisms Phi and Phiodd}}
\label{sec:proof1}

In what follows, we denote by $x$, $x^{-1}$, and $z$
the lists of variables $x_1,\ldots,x_n$,
$x_1^{-1},\ldots,x_n^{-1}$,
and $z_1,\ldots,z_n$, respectively, where $z_j=x_j+x_j^{-1}$.

The next proposition also appears in~\cite[eq. (4.4)]{Krattenthaler1998}.

\begin{prop}\label{prop:elem_x_recipx_symmetry}
For each $m$ in $\{0,\ldots,2n\}$,
\begin{equation}\label{eq:elem_x_xrecip_symmetry}
\elem_{2n-m}(x,x^{-1})=\elem_m(x,x^{-1}).
\end{equation}
\end{prop}

\begin{proof}
We know that the polynomial $f$
from Proposition~\ref{prop:palindromic_pol_defined_by_roots}
is palindromic.
This fact and Vieta's formula yield~\eqref{eq:elem_x_xrecip_symmetry}.
\end{proof}

The next formula was recently published by
Lachaud~\cite[Lemma~A.3]{Lachaud2016}.
We found it independently, but with exactly the same proof.
Therefore here we only give an idea of the proof.

\begin{prop}\label{prop:elem_sympl_via_elem_z}
For every $m$ in $\{0,\ldots,2n\}$,
\begin{equation}\label{eq:elem_sympl_via_elem_z}
\elem_m(x,x^{-1})
=\sum_{k=\max\{m-n,0\}}^{\lfloor m/2\rfloor}
\binom{n+2k-m}{k}\elem_{m-2k}(z).
\end{equation}
\end{prop}

\begin{proof}
From Propositions~\ref{prop:palindromic_pol_defined_by_roots},
\ref{prop:elem_x_recipx_symmetry}
and Vieta's formulas for the coefficients of
the polynomials $f$ and $g$,
\begin{equation}\label{eq:f_eq_g_through_elem_pol}
\sum_{m=0}^{2n}(-1)^m \elem_m(x,x^{-1})t^m
=\sum_{j=0}^n (-1)^{n-j} \elem_{n-j}(z) t^n (t+t^{-1})^j.
\end{equation}
Expanding $(t+t^{-1})^j$ by the binomial theorem
and matching the coefficient of $t^m$
yields~\eqref{eq:elem_sympl_via_elem_z}.
\end{proof}

For every $j$ in $\{1,\ldots,n\}$,
let $\proddif_j(z)$ be defined as follows:
\[
\proddif_j(z)
\eqdef\prod_{k\in\{1,\ldots,n\}\setminus\{j\}}(z_j-z_k).
\]
Denote by $\VandermondePol(z)$ the Vandermonde polynomial
in the variables $z_1,\ldots,z_n$:
\begin{equation}\label{eq:vanderDef}
\VandermondePol(z)
\eqdef
\det \bigl[ z_k^{n-j} \bigr]_{j,k=1}^n
=\prod_{j=1}^n \proddif_1(z_j,\ldots,z_n)
=\prod_{1\le j<k\le n}(z_j-z_k).
\end{equation}
Recall that
\begin{equation}\label{eq:h_via_geometric_progressions}
\hom_m(y_1,\ldots,y_p)
=\sum_{j=1}^p \frac{y_j^{m+p-1}}{\proddif_j(y_1,\ldots,y_p)}.
\end{equation}

\begin{lem}\label{lem:sum_powers_over_omega_low_degree}
Let $0\le s<n-1$. Then
\begin{equation}\label{eq:sum_powers_over_omega_low_degree}
\sum_{j=1}^n\frac{z_j^s}{\proddif_j(z)}=0.
\end{equation}
\end{lem}

\begin{proof}
It is easy to see that the left-hand side of~\eqref{eq:sum_powers_over_omega_low_degree}
equals $\det(A(z))/\VandermondePol(z)$,
where $A(z)$ is the matrix with the entries
\[
A_{j,k}(z)\eqdef
\begin{cases}
z_k^s, & j=1 \\
z_k^{n-j}, & 2\le j\le n.
\end{cases}
\]
Since the first row of $A(z)$ coincides with the $(n-s)$th,
$\det(A(z))=0$.
\end{proof}

\begin{prop}\label{prop:hom_sympl_via_hom_z}
We have the identity
\begin{equation}\label{eq:hom_sympl_via_hom_z}
\hom_m(x,x^{-1})
= \sum_{k=0}^{\lfloor m/2\rfloor}
(-1)^k\binom{n + m - k - 1}{k}\hom_{m-2k}(z_1,\ldots,z_n).
\end{equation}
\end{prop}

\begin{proof}
Apply \eqref{eq:h_via_geometric_progressions}
with the variables $x,x^{-1}$:
\begin{equation}\label{eq:hom_x_recip_x_via_geom_progr}
\hom_m(x,x^{-1})
=\sum_{j=1}^n \frac{x_j^{2n+m-1}}{\proddif_j(x,x^{-1})}
+\sum_{j=1}^n \frac{x_j^{-2n-m+1}}{\proddif_{n+j}(x,x^{-1})}.
\end{equation}
Since
\[
(x_j-x_k)(x_j-x_k^{-1})=x_j(z_j-z_k),\qquad
(x_j^{-1}-x_k)(x_j^{-1}-x_k^{-1})=x_j^{-1}(z_j-z_k),
\]
the denominators in~\eqref{eq:hom_x_recip_x_via_geom_progr}
can be written as
\[
\proddif_j(x,x^{-1})
=(x_j-x_j^{-1}) x_j^{n-1} \proddif_j(z),\qquad
\proddif_{n+j}(x,x^{-1})
=-(x_j-x_j^{-1}) x_j^{-n+1} \proddif_j(z).
\]
Applying \eqref{eq:ChebU_main_property} we arrive at
\begin{equation}\label{eq:hom_sympl_via_quotients_U}
\hom_m(x,x^{-1})
=\sum_{j=1}^n \frac{\cU_{m+n-1}(z_j/2)}{\proddif_j(z)}.
\end{equation}
From this and \eqref{eq:U_explicit},
\[
\hom_m(x, x^{-1})
=\sum_{j=1}^n \frac{\cU_{n+m-1}(z_j/2)}{\proddif_j(z)}
=\sum_{k=0}^{\lfloor (n+m-1)/2\rfloor}
(-1)^k \binom{n+m-k-1}{k}
\sum_{j=1}^n \frac{z_j^{m+n-1-2k}}{\proddif_j(z)}.
\]
If $k>\lfloor m/2\rfloor$, then $m+n-1-2k<n-1$,
and the inner sum is zero by
Lemma~\ref{lem:sum_powers_over_omega_low_degree}.
So, the range of $k$ in the outer sum can be restricted to
$\{0,\ldots,\lfloor m/2\rfloor\}$.
Applying~\eqref{eq:h_via_geometric_progressions}
yields~\eqref{eq:palhom_via_hom}.
\end{proof}

\begin{prop}\label{prop:powersum_sympl_via_z}
For every $m$ in $\bN_0$,
\begin{equation}\label{eq:powersum_sympl_via_powersum_z}
\powersum_m(x,x^{-1})
=
\begin{cases}
\displaystyle m\sum_{j=0}^{\lfloor m/2 \rfloor}
\frac{(-1)^j}{m-j}\binom{m-j}{j}\powersum_{m-2j}(z_1,\ldots,z_n), & m\in\bN;\\[0.5ex]
2\powersum_0(z_1,\ldots,z_n), & m=0.
\end{cases}
\end{equation}
\end{prop}

\begin{proof}
The result is trivial for $m=0$.
Suppose $m\in\bN$.
By~\eqref{eq:ChebT_main_property},
\[
\powersum_m(x,x^{-1})
=\sum_{k=1}^{n}(x_k^m+x_k^{-m})
=\sum_{k=1}^n 2\cT_m\bigl((x_k+x_k^{-1})/2\bigr),
\]
i.e.
\begin{equation}\label{eq:powersum_sympl_via_ChebT_z}
\powersum_m(x,x^{-1})=\sum_{k=1}^n2\cT_m\left(z_k/2\right).
\end{equation}
Furthermore, applying~\eqref{eq:T_explicit},
\begin{align*}
\powersum_m(x,x^{-1})
&=m\sum_{k=1}^n \sum_{j=0}^{\lfloor m/2\rfloor}
\frac{(-1)^j}{m-j}\binom{m-j}{j}z_k^{m-2j}
=m\sum_{j=0}^{\lfloor m/2\rfloor}
\frac{(-1)^j}{m-j}\binom{m-j}{j}\sum_{k=1}^n z_k^{m-2j}
\\
&=m\sum_{j=0}^{\lfloor m/2\rfloor}
\frac{(-1)^j}{m-j}\binom{m-j}{j}\powersum_{m-2j}(z).
\qedhere
\end{align*}
\end{proof}

\begin{proof}[Proof of Theorem~\ref{thm:general}]
Since the set $\{\elem_m\}_{m=1}^{2n}$
generates the unital algebra $\Sym_{2n}$,
the existence in Theorem~\ref{thm:general}
follows from Proposition~\ref{prop:elem_sympl_via_elem_z}.
Similarly, the existence also follows
from Proposition~\ref{prop:hom_sympl_via_hom_z}
and from Proposition~\ref{prop:powersum_sympl_via_z}.

For the uniqueness, suppose that $Q_1,Q_2\in\Sym_n$
such that
\[
Q_1\left(x_1+x_1^{-1},\ldots,x_n+x_n^{-1}\right)
=Q_2\left(x_1+x_1^{-1},\ldots,x_n+x_n^{-1}\right).
\]
Consider the difference $Q_3=Q_1-Q_2$.
The assumptions on $Q_1$ and $Q_2$ imply that $Q_3\in\Sym_n$
and the expression
\begin{equation}\label{eq:Q3}
Q_3\left(x_1+x_1^{-1},\ldots,x_n+x_n^{-1}\right)
\end{equation}
is zero.
If $Q_3$ is a non-zero polynomial
and $c z_1^{\al_1}\cdots z_n^{\al_n}$
is one of the leading terms of $Q_3$,
then it is easy to see that the expression~\eqref{eq:Q3}
contains the summand $c x_1^{\al_1}\cdots x_n^{\al_n}$.
This contradiction shows that $Q_3$ has to be the zero polynomial.
\end{proof}

\begin{proof}[Proof of Theorem~\ref{thm:general_odd}]
In order to prove the existence, define $R\in\Sym_{2n}$ by
\[
R(y_1,\ldots,y_{2n})\eqdef P(y_1,\ldots,y_{2n},1).
\]
Applying Theorem~\ref{thm:general} to $R$
we get $Q$ in $\Sym_n$ which has the desired property:
\begin{align*}
Q(x_1+x_1^{-1},\ldots,x_n+x_n^{-1})
&=R(x_1,\ldots,x_n,x_1^{-1},\ldots,x_n^{-1})
\\
&=P(x_1,\ldots,x_n,x_1^{-1},\ldots,x_n^{-1},1).
\end{align*}
The uniqueness also reduces
to the uniqueness in Theorem~\ref{thm:general}.
\end{proof}

\begin{rem}\label{rem:Phi_explicit}
The two-valued inverse Dickson--Zhukovsky transform,
i.e. the pair of the solutions of the equation $t+t^{-1}=u$,
is given by $t=(u\pm\sqrt{u^2-4})/2$.
So, $\Phi_n$ acts by the following explicit rule:
\begin{multline}\label{eq:Phi_n_via_inverse_Zhukovsky}
(\Phi_n(P))(z_1,\ldots,z_n)
=P\Biggl(\frac{z_1+\sqrt{z_1^2-4}}{2},\ldots,\frac{z_n+\sqrt{z_n^2-4}}{2},\\
\frac{z_1-\sqrt{z_1^2-4}}{2},\ldots,\frac{z_n-\sqrt{z_n^2-4}}{2}\Biggr),
\end{multline}
and Theorem~\ref{thm:general} ensures that the right-hand side
of~\eqref{eq:Phi_n_via_inverse_Zhukovsky}
is a symmetric polynomial in $z_1,\ldots,z_n$.
\end{rem}

\begin{example}\label{example:not_homogeneous}
Let $n=2$ and
\[
P(x_1,x_2,x_3,x_4)
=x_1^2+x_2^2+x_3^2+x_4^2
+5x_1 x_2+5x_1 x_3+5x_1 x_4+5x_2 x_3+5x_2 x_4+5x_3 x_4.
\]
Then
\[
P(x_1,x_2,x_1^{-1},x_2^{-1})
=\left(x_1+x_1^{-1}\right)^2+\left(x_2+x_2^{-1}\right)^2
+5\left(x_1+x_1^{-1}\right)\left(x_2+x_2^{-1}\right)+6,
\]
i.e.
\[
\Phi_2(P)(z_1,z_2) = z_1^2 + z_2^2 + 5z_1z_2 + 6.
\]
In this example $\Phi_2(P)$ is not homogeneous, though $P$ is homogeneous.
\end{example}

\begin{example}\label{example:not_injective}
For a general $n$, put
\begin{align*}
P_1(x_1,\ldots,x_{2n})
&=\elem_{2n}(x_1,\ldots,x_{2n})
=\prod_{j=1}^{2n} x_j,
&
P_2(x_1,\ldots,x_{2n})=1,
\\
P_3(x_1,\ldots,x_{2n+1})
&=\elem_{2n+1}(x_1,\ldots,x_{2n+1})
=\prod_{j=1}^{2n+1} x_j,
&
P_4(x_1,\ldots,x_{2n+1})=1.
\end{align*}
Then $P_1(x,x^{-1})=P_2(x,x^{-1})=1$
and $P_3(x,x^{-1},1)=P_4(x,x^{-1},1)=1$, i.e.
\[
\Phi_n(P_1)=\Phi_n(P_2),\qquad
\Phi^{\odd}_n(P_3)=\Phi^{\odd}_n(P_4).
\]
This example shows that
$\Phi_n$ and $\Phi^{\odd}_n$ are not injective.
\end{example}

\section{\texorpdfstring{Symplectic and orthogonal Schur polynomials\\
and Schur--Chebyshev quotients}{Symplectic and orthogonal Schur polynomials
and Schur--Chebyshev quotients}}
\label{sec:Schur_Chebyshev_quotients}

In this section we recall various equivalent formulas
for symplectic and orthogonal Schur polynomials
and relate them with Schur--Chebyshev quotients.

The symplectic and orthogonal Schur functions
$\sympl_\la$ and $\orth_\la$
appear naturally in  representation theory
in relation with the symplectic and orthogonal groups.
See~\cite[Section~24 and Appendix~A]{FultonHarris1991}
for most of the following formulas
for $\sympl_\la$ and $\orth_\la$.
An additional information can be found in~\cite{King1976}.

For every $\la$ in $\Partitions$,
the symplectic Schur function $\sympl_\la$
and the orthogonal Schur function $\orth_\la$
are elements of $\Sym$
defined by the following analogs of Jacobi--Trudi identities:
\begin{align}
\label{eq:sp_JT1}
\sympl_\la
&\eqdef \frac{1}{2} \det \bigl[
\hom_{\la_j-j+k}+\hom_{\la_j-j-k+2}
\bigr]_{j,k=1}^{\ell(\la)},
\\
\label{eq:so_JT1}
\orth_\la
&\eqdef \det \bigl[
\hom_{\la_j-j+k}-\hom_{\la_j-j-k}
\bigr]_{j,k=1}^{\ell(\la)}.
\end{align}
They also satisfy the following analogs of dual Jacobi--Trudi identities:
\begin{align}
\label{eq:sp_JT2}
\sympl_\la
&= \det \bigl[
\elem_{\la^{\prime}_j-j+k}-\elem_{\la^{\prime}_j-j-k}
\bigr]_{j,k=1}^{\la_1},
\\
\label{eq:so_JT2}
\orth_\la
&= \frac{1}{2} \det \bigl[
\elem_{\la^{\prime}-j+k} + \elem_{\la^{\prime}-j-k+2}
\bigr]_{j,k=1}^{\la_1}.
\end{align}
The functions $\sympl_\la$ and $\orth_\la$
are traditionally evaluated at the \emph{even symplectic alphabet}
\[
(x,x^{-1})=(x_1,\ldots,x_n,x_1^{-1},\ldots,x_n^{-1}).
\]
Since $\Phi_n$ is a morphism of algebras,
\eqref{eq:sp_JT1}--\eqref{eq:so_JT2}
imply the following analogs of Jacobi--Trudi identities
for $\spz_\la$ and $\oz_\la$:
\begin{align}
\label{eq:JT1_spz}
\spz_\la(z)
&=\frac{1}{2}
\det\bigl[
\homz_{\la_j-j+k}(z)+\homz_{\la_j-j-k+2}(z)\bigr]_{j,k=1}^{\ell(\la)},
\\
\label{eq:JT2_spz}
\spz_\la(z)
&=\det\bigl[
\elemz_{\la'_j-j+k}(z)-\elemz_{\la'_j-j-k}(z)\bigr]_{j,k=1}^{\la_1},
\\
\label{eq:JT1_oz}
\oz_\la(z)
&=\det\bigl[
\homz_{\la_j-j+k}(z)-\homz_{\la_j-j-k}(z)\bigr]_{j,k=1}^{\ell(\la)},
\\
\label{eq:JT2_oz}
\oz_\la(z)
&=\frac{1}{2}\det\bigl[
\elemz_{\la'_j-j+k}(z)+\elemz_{\la'_j-j-k+2}(z)\bigr]_{j,k=1}^{\la_1}.
\end{align}
There are bialternant formulas
for $\sympl_\la(x,x^{-1})$
and $\orth_\la(x,x^{-1})$:
\begin{align}
\label{eq:sp_bialternant}
\sympl_\la(x,x^{-1})
&=\frac{\det\bigl[x_k^{\la_j+n-j+1}-x_k^{-(\la_j+n-j+1)}\bigr]_{j,k=1}^n}%
{\det\bigl[x_k^{n-j+1}-x_k^{-(n-j+1)}\bigr]_{j,k=1}^n},
\\
\label{eq:so_bialternant}
\orth_\la(x,x^{-1})
&=\frac{\det\bigl[\zeta_{\la_j+n-j}(x_k)\bigr]_{j,k=1}^n}%
{\det\bigl[\zeta_{n-j}(x_k)\bigr]_{j,k=1}^n},
\end{align}
where $\zeta_m(u)\coloneqq u^m+u^{-m}$ if $m>0$ and $\zeta_0(u)\coloneqq 1$.

Characters of the odd orthogonal groups
are related with the functions $\orth_\la$
evaluated at the \emph{odd symplectic alphabet}
\[
(x,x^{-1},1)=(x_1,\ldots,x_n,x_1^{-1},\ldots,x_n^{-1},1).
\]
It is known that
\begin{align}
\label{eq:orth_odd_bialternant}
\orth_\la(t_1^2,\ldots,t_n^2,t_1^{-2},\ldots,t_n^{-2},1)
&=\frac{\det\bigl[t_k^{2\la_j+2n-2j+1}-t_k^{-(2\la_j+2n-2j+1)}\bigr]_{j,k=1}^n}%
{\det\bigl[t_k^{2n-2j+1}-t_k^{-(2n-2j+1)}\bigr]_{j,k=1}^n}.
\end{align}

Okada~\cite{Okada2019} proved the following
bialternant formula for $\sympl_\la$
evaluated at the generalized odd symplectic alphabet:
if $\la\in\Partitions_{n+1}$, then
\begin{equation}\label{eq:sympl_odd_bialternant}
\sympl_\la(t_1^2,\ldots,t_n^2,t_1^{-2},\ldots,t_n^{-2},t_{n+1})
=\frac{\det A_\la(t_1,\ldots,t_n,1)}%
{\det A_\varnothing(t_1,\ldots,t_n,1)},
\end{equation}
where $(A_\la(t_1,\ldots,t_n,t_{n+1}))_{j,k}$ is defined as
\[
\begin{cases}
(t_k^{2\la_j+2n-2j+4} - t_k^{-(2\la_j+2n-2j+4)})
- t_{n+1}^{-1} (t_k^{2\la_j+2n-2j+2} - t_k^{-(2\la_j+2n-2j+2)}) &
\text{if}\ 1\leq k\leq n,\\
t_{n+1}^{2\la_j+2n-2j+2} &
\text{if}\ k = n+1.
\end{cases}
\]
The bialternant formulas~\eqref{eq:sp_bialternant},
\eqref{eq:so_bialternant}, \eqref{eq:orth_odd_bialternant}, and
\eqref{eq:sympl_odd_bialternant}
are natural to rewrite in terms of the Chebyshev polynomials,
as certain ``Schur--Chebyshev quotients''.

\begin{prop}\label{prop:det_TUVW_with_zero_partition}
We have that
\begin{align*}
\det\bigl[\cTmonic_{n-j}(z_k)\bigr]_{j,k=1}^n
&=\det \bigl[ \cU_{n-j}(z_k/2) \bigr]_{j,k=1}^n
=\det \bigl[ \cV_{n-j}(z_k/2)\bigr]_{j,k=1}^n
\\
&=\det \bigl[ \cW_{n-j}(z_k/2)\bigr]_{j,k=1}^n
=\det \bigl[ \cU^{(1)}_{n-j}(z_k/2)\bigr]_{j,k=1}^n
=\VandermondePol(z).
\end{align*}
\end{prop}

\begin{proof}
We sketch how to prove the  formula for
$\det\bigl[\cU_{n-j}(z_k/2)\bigr]_{j,k=1}^n$;
the other four determinants can be computed in a similar way.
We apply the same ideas used to compute
the classical Vandermonde determinant $\VandermondePol(z)$.
One way is to use elementary transformations of the determinants
and to obtain the following recursive formula:
\begin{equation*}
\det\left[\cU_{n-j}(z_k/2)\right]_{j,k=1}^n
=\proddif_1(z_1,\ldots,z_n)
\det\left[\cU_{n-j}(z_{k+1}/2)\right]_{j,k=1}^{n-1}.
\end{equation*}
Another way is to notice that the determinant
of $\bigl[\cU_{n-j}(z_k/2)\bigr]_{j,k=1}^n$ vanishes
when any two of the variables $z_1,\ldots,z_n$ coincide,
therefore it must be a multiple of $\VandermondePol(z)$,
i.e.\ there exists a polynomial $C(z)$ such that
\begin{equation*}
\det \left[\cU_{n-j}(z_k/2)\right]_{j,k=1}^n
=C(z)\VandermondePol(z).
\end{equation*}
Comparing the coefficient of the term
$z_1^{n-1}z_2^{n-2}\dotsm z_n^0$
yields $C(z)=1$.
\end{proof}

\begin{prop}\label{prop:sympl_orth_bialternant_via_z}
Let $z_k=x_k+x_k^{-1}$ for every $k$ in $\{1,\ldots,n\}$.
Then
\begin{align}
\label{eq:sympl_sympl_via_uquotient}
\sympl_\la(x,x^{-1})
&=\frac{\det\bigl[\cU_{\la_j+n-j}(z_k/2)\bigr]_{j,k=1}^n}%
{\det\bigl[\cU_{n-j}(z_k/2)\bigr]_{j,k=1}^n}
=\frac{\det\bigl[\cU_{\la_j+n-j}(z_k/2)\bigr]_{j,k=1}^n}%
{\VandermondePol(z)},
\\
\label{eq:orth_sympl_via_tquotient}
\orth_\la(x,x^{-1})
&=\frac{\det\bigl[\cTmonic_{\la_j+n-j}(z_k)\bigr]_{j,k=1}^n}%
{\det\bigl[\cTmonic_{n-j}(z_k)\bigr]_{j,k=1}^n}
=\frac{\det\bigl[\cTmonic_{\la_j+n-j}(z_k)\bigr]_{j,k=1}^n}%
{\VandermondePol(z)},
\\
\label{eq:sympl_sympl_odd_via_usumquotient}
\sympl_\la(x,x^{-1},1)
&=\frac{\det\bigl[\cU^{(1)}_{\la_j+n-j}(z_k/2)\bigr]_{j,k=1}^n}%
{\det\bigl[\cU^{(1)}_{n-j}(z_k/2)\bigr]_{j,k=1}^n}
=\frac{\det\bigl[\cU^{(1)}_{\la_j+n-j}(z_k/2)\bigr]_{j,k=1}^n}%
{\VandermondePol(z)},
\\
\label{eq:orth_symplodd_via_wquotient}
\orth_\la(x,x^{-1},1)
&=\frac{\det\bigl[\cW_{\la_j+n-j}(z_k/2)\bigr]_{j,k=1}^n}%
{\det\bigl[\cW_{n-j}(z_k/2)\bigr]_{j,k=1}^n}
=\frac{\det\bigl[\cW_{\la_j+n-j}(z_k/2)\bigr]_{j,k=1}^n}%
{\VandermondePol(z)},
\\
\orth_\la(x,x^{-1},-1)
&=(-1)^{|\la|}\orth_\la(-x,-x^{-1},1)
\notag
\\
\label{eq:orth_negsymplodd_via_vquotient}
&=\frac{\det\bigl[\cV_{\la_j+n-j}(z_k/2)\bigr]_{j,k=1}^n}%
{\det\bigl[\cV_{n-j}(z_k/2)\bigr]_{j,k=1}^n}
=\frac{\det\bigl[\cV_{\la_j+n-j}(z_k/2)\bigr]_{j,k=1}^n}%
{\VandermondePol(z)}.
\end{align}
\end{prop}

\begin{proof}
Let us prove the first equality in~\eqref{eq:sympl_sympl_via_uquotient}.
Use the bialternant formula~\eqref{eq:sp_bialternant}
and the property~\eqref{eq:ChebU_main_property}
of the polynomials $\cU$:
\[
\sympl_\la(x,x^{-1})
=\frac{\prod_{k=1}^n (x_k-x_k^{-1})
\det\bigl[\cU_{\la_j+n-j}(z_k/2)\bigr]_{j,k=1}^n}%
{\prod_{k=1}^n (x_k-x_k^{-1})
\det\bigl[\cU_{n-j}(z_k/2)\bigr]_{j,k=1}^n}
=\frac{\det\bigl[\cU_{\la_j+n-j}(z_k/2)\bigr]_{j,k=1}^n}%
{\det\bigl[\cU_{n-j}(z_k/2)\bigr]_{j,k=1}^n}.
\]
The second equality in~\eqref{eq:sympl_sympl_via_uquotient}
follows from the first one transforming the denominator
by~Proposition~\ref{prop:det_TUVW_with_zero_partition}.
In a similar way, \eqref{eq:orth_sympl_via_tquotient}
follows from~\eqref{eq:so_bialternant}
and~\eqref{eq:ChebT_main_property},
and \eqref{eq:orth_symplodd_via_wquotient}
follows from~\eqref{eq:orth_odd_bialternant}
and~\eqref{eq:ChebW_main_property}.
Finally, \eqref{eq:orth_negsymplodd_via_vquotient}
is a consequence of~\eqref{eq:orth_symplodd_via_wquotient}
and \eqref{eq:V_and_W}.

Let us prove \eqref{eq:sympl_sympl_odd_via_usumquotient}
using Okada's formula~\eqref{eq:sympl_odd_bialternant}
with $\la_{n+1}=0$, $x_k=t_k^2$ ($1\le k\le n$) and $t_{n+1}=1$.
For $1\le k\le n$,
\begin{align*}
A_\la(t_1,\ldots,t_n,1)_{j,k}
&=(x_k^{\la_j+n-j+2} - x_k^{-(\la_j+n-j+2)})
- (x_k^{\la_j+n-j+1} - x_k^{-(\la_j+n-j+1)})
\\
&=(x_k-x_k^{-1})
(\cU_{\la_j+n-j+1}(z_j/2)-\cU_{\la_j+n-j}(z_j/2)
\\
&=(x_k-x_k^{-1})\cV_{\la_j+n-j+1}(z_j/2).
\end{align*}
In particular, if $1\le k\le n$,
then the entry $(n+1,k)$ is $(x_k-x_k^{-1})\cV_0(z_j/2)$.
After factoring $x_k-x_k^{-1}$ from the $k$th column
($1\le k\le n$),
\[
\det A_\la(t_1,\ldots,t_n,1)
=
\left(\prod_{k=1}^n (x_k-x_k^{-1})\right)
\begin{vmatrix}
\bigl[\cV_{\la_j+n+1-j}(z_k/2)\bigr]_{j,k=1}^n &
\bigl[1 \bigr]_{j=1}^n \\[1ex]
\bigl[1,\ldots,1\bigr] & 1
\end{vmatrix}.
\]
Now from each column $1,\ldots,n$ we subtract the column $n+1$,
then we expand the determinant by the last row
and use the formula~\eqref{eq:sumU_via_V}.
\begin{align*}
\det A_\la(t_1,\ldots,t_n,1)
&=\left(\prod_{k=1}^n (x_k-x_k^{-1})\right)
\det\bigl[\cV_{\la_j+n+1-j}(z_k/2)-1\bigr]_{j,k=1}^n
\\
&=\left(\prod_{k=1}^n (x_k-x_k^{-1})\right)
\left(\prod_{k=1}^n (z_k-2)\right)
\det\bigl[\cU^{(1)}_{\la_j+n-j}(z_k/2)\bigr]_{j,k=1}^n.
\end{align*}
Thus,
\[
\spz_\la(x,x^{-1},1)
=\frac{\det A_\la(t_1,\ldots,t_n,1)}%
{\det A_\varnothing(t_1,\ldots,t_n,1)}
=\frac{\det\bigl[\cU^{(1)}_{\la_j+n-j}(z_k/2)\bigr]_{j,k=1}^n}%
{\det\bigl[\cU^{(1)}_{n-j}(z_k/2)\bigr]_{j,k=1}^n}.
\]
Taking into account Proposition~\ref{prop:det_TUVW_with_zero_partition},
we obtain~\eqref{eq:sympl_sympl_odd_via_usumquotient}.
\end{proof}

Theorem~\ref{thm:spz_oz_bialternant} follows
from Proposition~\ref{prop:sympl_orth_bialternant_via_z}
and the definitions of $\Phi_n$ and $\Phiodd_n$.
Duplication formulas for the Chebyshev polynomials
lead to another equivalent form
of the bialternant formulas from
Proposition~\ref{prop:sympl_orth_bialternant_via_z}.

\begin{prop}\label{prop:spz_oz_duplication}
For every $\la$ in $\Partitions_n$,
\begin{align}
\label{eq:spz_duplication}
\spz_\la(u_1^2-2,\ldots,u_n^2-2)
&=\frac{\det\bigl[ \cU_{2\la_j+2n-2j+1}(u_k/2) \bigr]_{j,k=1}^n}%
{\det\bigl[ \cU_{2n-2j+1}(u_k/2) \bigr]_{j,k=1}^n},
\\
\label{eq:oz_duplication}
\oz_\la(u_1^2-2,\ldots,u_n^2-2)
&=\frac{\det\bigl[ \cTmonic_{2\la_j+2n-2j}(u_k) \bigr]_{j,k=1}^n}%
{\det\bigl[ \cTmonic_{2n-2j}(u_k) \bigr]_{j,k=1}^n},
\\
\label{eq:oz_odd_duplication}
\oz^{\odd}_\la(u_1^2-2,\ldots,u_n^2-2)
&=\frac{\det\bigl[ \cU_{2\la_j+2n}(u_k/2) \bigr]_{j,k=1}^n}%
{\det\bigl[ \cU_{2n-2j}(u_k/2) \bigr]_{j,k=1}^n},
\\
\label{eq:oz_odd_neg_duplication}
(-1)^{|\la|}\oz^{\odd}_\la(-u_1^2+2,\ldots,-u_n^2+2)
&=\frac{\det\bigl[ \cT_{2\la_j+2n-2j+1}(u_k/2) \bigr]_{j,k=1}^n}%
{\det\bigl[ \cT_{2n-2j+1}(u_k/2) \bigr]_{j,k=1}^n}.
\end{align}
\end{prop}

\begin{proof}
Apply~\eqref{eq:spz_as_uquotient} and~\eqref{eq:duplication_U}:
\begin{align*}
\spz_\la(u_1^2-2,\ldots,u_n^2-2)
&=
\frac{\det\bigl[u_k\,\cU_{2\la_j+2n-2j+1}(u_k/2)\bigr]_{j,k=1}^n}%
{\det\bigl[u_k\,\cU_{2n-2j+1}(u_k/2)\bigr]_{j,k=1}^n}
\\
&=
\frac{\left(\prod_{k=1}^n u_k\right)
\det\bigl[\cU_{2\la_j+2n-2j+1}(u_k/2)\bigr]_{j,k=1}^n}%
{\left(\prod_{k=1}^n u_k\right)
\det\bigl[\cU_{2n-2j+1}(u_k/2)\bigr]_{j,k=1}^n}.
\end{align*}
After canceling~$\prod_{k=1}^n u_k$
we obtain~\eqref{eq:spz_duplication}.
Similarly, formula~\eqref{eq:oz_duplication}
follows from~\eqref{eq:oz_as_tquotient}
and~\eqref{eq:duplication_T},
\eqref{eq:oz_odd_duplication}
follows from~\eqref{eq:ozodd_as_wquotient}
and~\eqref{eq:duplication_W},
and~\eqref{eq:oz_odd_neg_duplication}
follows from~\eqref{eq:ozodd_neg_as_vquotient}
and~\eqref{eq:duplication_V}.
\end{proof}

\section{\texorpdfstring{Properties of
$\boldsymbol{\elemz}$, $\boldsymbol{\homz}$,
and $\boldsymbol{\powersumz}$}%
{Properties of ez, hz, and pz}}
\label{sec:properties_pal_e_h_p}

Denote by $\Hpal(z)(t)$ the generating series
of the sequence $(\homz_k(z))_{k=0}^\infty$
and by $\Epal(z)(t)$ the generating series
of the sequence $(\elemz_k(z))_{k=0}^\infty$:
\[
\Hpal(z)(t)
\eqdef\sum_{k=0}^\infty \homz_k(z)t^k,\qquad
\Epal(z)(t)
\eqdef\sum_{k=0}^\infty \elemz_k(z)t^k
=\sum_{k=0}^{2n}\elemz_k(z)t^k.
\]

\begin{prop}\label{prop:Epal_and_Hpal}
The generating series are given by
\begin{align}
\label{eq:Epal}
\Epal(z)(t)&=\prod_{j=1}^n (1+z_j t+t^2),
\\
\label{eq:Hpal}
\Hpal(z)(t)&=\prod_{j=1}^n \frac{1}{1-z_j t+t^2}.
\end{align}
\end{prop}

\begin{proof}
Given a list of variables $a=(a_1,\ldots,a_p)$,
$\Elem(a)(t)$ stands for the generating series
of the sequence $(\elem_k(a))_{k=0}^\infty$
and $\Hom(a)(t)$ for the generating series
of the sequence $(\hom_k(a))_{k=0}^\infty$:
\[
\Elem(a)(t)
=\sum_{k=0}^\infty \elem_k(a) t^k
=\sum_{k=0}^p \elem_k(a) t^k
=\prod_{j=1}^p (1+a_p t),\qquad
\Hom(a)(t)
=\sum_{k=0}^\infty \hom_k(a) t^k.
\]
Using the substitutions $z_j=x_j+x^{-1}_j$ we obtain
\begin{align*}
\Epal(z)(t)
&=\sum_{k=0}^\infty \elemz_k(z)t^k
=\sum_{k=0}^\infty \elem_k(x,x^{-1})t^k
=\Elem(x,x^{-1})(t)
\\
&=\prod_{j=1}^n \bigl((1+x_j t)(1+x_j^{-1}t)\bigr)
=\prod_{j=1}^n (1+z_j t+t^2).
\end{align*}
Since the series $\Hom(x,x^{-1})(t)$
is the reciprocal of $\Elem(x,x^{-1})(-t)$,
the series $\Hpal(z)(t)$ is the reciprocal of $\Epal(z)(t)$.
\end{proof}

\subsection*{Formulas for $\boldsymbol{\elemz}$}

Now we rewrite Proposition~\ref{prop:elem_x_recipx_symmetry}
using the notation $\elemz$.

\begin{prop}\label{prop:elemz_symmetry}
For each $m$ in $\{0,\ldots,2n\}$,
\begin{equation}\label{eq:elemz_symmetry}
\elemz_{2n-m}(z)=\elemz_m(z).
\end{equation}
\end{prop}

Next we prove Theorem~\ref{thm:pal_elem}.
Formula~\eqref{eq:elemz_via_elem}
follows from Proposition~\ref{prop:elem_sympl_via_elem_z},
and~\eqref{eq:epal_as_tquotient}
follows from~\eqref{eq:JT2_oz} with $\la=(1^m)$.

\begin{proof}[Proof of~\eqref{eq:elem_via_elemz}] 
Let $f$ and $g$ be as in Propositions~\ref{prop:palindromic_pol_defined_by_roots},
\ref{prop:palindromic_pol_through_ChebT}.
By~\eqref{eq:T_explicit_via_coefT},
\begin{align*}
g(t+t^{-1})
&=
(-1)^n \elem_n(x,x^{-1})+\sum_{j=1}^{n}(-1)^{n-j} \elem_{n-j}(x,x^{-1})
2\cT_{j}\left(\frac{t+t^{-1}}{2}\right)
\\
&=\sum_{j=0}^n \sum_{k=0}^{\lfloor j/2\rfloor}
(-1)^{n-j} \elem_{n-j}(x,x^{-1})(-1)^k\coefT_{j,k} (t+t^{-1})^{j-2k}
\\
&=\sum_{k=0}^{\lfloor n/2\rfloor}
\sum_{j=2k}^{n} (-1)^{n-j+k} \coefT_{j,k}
\elem_{n-j}(x,x^{-1}) (t+t^{-1})^{j-2k}
\\
&=\sum_{k=0}^{\lfloor n/2\rfloor}\sum_{m=2k}^n (-1)^{m-k}
\coefT_{n-m+2k,k} \elemz_{m-2k}(z) (t+t^{-1})^{n-m}
\\
&=\sum_{m=0}^n (-1)^m
\left(\sum_{k=0}^{\lfloor m/2\rfloor}
(-1)^k \coefT_{n-m+2k,k}\elemz_{m-2k}(z)\right) (t+t^{-1})^{n-m}.
\end{align*}
On the other hand, by Vieta's formula 
\[
g(t+t^{-1})
=\sum_{m=0}^n (-1)^m \elem_m(z)(t+t^{-1})^{n-m}.
\]
Matching the coefficients of the powers of $(t+t^{-1})$
we obtain~\eqref{eq:elem_via_elemz}.
\end{proof}

\begin{proof}[Proof of \eqref{eq:elemz_via_uquotients}]
If $\la=(1^q)$, then $\la'=(q)$,
and~\eqref{eq:JT2_spz} takes the form
\begin{equation}\label{eq:spz_1_via_elemz}
\spz_{(1^q)}(z)
=\elemz_q(z)-\elemz_{q-2}(z).
\end{equation}
In particular, $\spz_{(1)}(z)=\elemz_1(z)$
and $\spz_{(0)}(z)=\elemz_0(z)$.
Now we write $\elemz_m(z)$
as a telescopic sum and apply~\eqref{eq:spz_1_via_elemz}:
\[
\elemz_m(z)
=\sum_{k=0}^{\lfloor m/2\rfloor-1}
(\elemz_{m-2k}(z)-\elemz_{m-2k-2}(z))
+\elemz_{m-2\lfloor m/2\rfloor}(z)
=\sum_{k=0}^{\lfloor m/2\rfloor}
\spz_{(1^{m-2k})}(z).
\qedhere
\]
\end{proof}

In the proposition below we abbreviate
$(z_1,\ldots,z_n,z_{n+1})$ as $(z,z_{n+1})$.

\begin{prop}\label{prop:elemz_sum_dif}
For every $m$ in $\bN_0$,
\begin{align}
\label{eq:sum_elemz}
&\elemz_{m+2}(z,z_{n+1})
=\elemz_{m+2}(z)+z_{n+1}\elemz_{m+1}(z)+\elemz_m(z),
\\
\label{eq:dif_elemz}
&\elemz_{m+1}(z,z_{n+1})-\elemz_{m+1}(z,z_{n+2})
=(z_{n+1}-z_{n+2})\elemz_m(z).
\end{align}
\end{prop}

\begin{proof}
By~\eqref{eq:Epal},
\begin{align}
\label{eq:Epal_recur}
&\Epal(z,z_{n+1})(t)
=(1+z_{n+1}t+t^2)\Epal(z)(t),
\\
\label{eq:Epal_dif}
&\Epal(z,z_{n+1})(t)-\Epal(z,z_{n+2})(t)
=(z_{n+1}-z_{n+2})t\Epal(z)(t).
\end{align}
Matching the coefficient of $t^m$ in \eqref{eq:Epal_recur}
yields~\eqref{eq:sum_elemz},
and matching the coefficient of $t^m$ in \eqref{eq:Epal_dif}
yields~\eqref{eq:dif_elemz}.
\end{proof}

\subsection*{Formulas for $\boldsymbol{\homz}$}

In this subsection we prove Theorem~\ref{thm:pal_hom}
and some simple relations between the expressions $\homz$
with different lists of arguments
(Proposition~\ref{prop:palhom_sum_dif}).
Notice that~\eqref{eq:palhom_as_sum_of_U}
is another form of~\eqref{eq:hom_sympl_via_quotients_U},
and~\eqref{eq:palhom_via_hom}
is another form of~\eqref{eq:hom_sympl_via_hom_z}.

\begin{lem}\label{lem:uquotient_particular}
Let $s\in\bN_0$. Then
\begin{equation}\label{eq:uquotient_particular}
\frac{
\det
\begin{bmatrix}
\cU_{n+s-1}(z_1/2) & \ldots & \cU_{n+s-1}(z_n/2)\\
\cU_{n-2}(z_1/2) & \ldots & \cU_{n-2}(z_n/2) \\
\ldots & \ldots & \ldots \\
\cU_0(z_1/2) & \ldots & \cU_0(z_n/2)
\end{bmatrix}
}{\VandermondePol(z)}
=\sum_{j=1}^n\frac{\cU_s(z_j/2)}{\proddif_j(z)}.
\end{equation}
\end{lem}

\begin{proof}
Expand the determinant in the numerator along the first row
and simplify the cofactors using
Proposition~\ref{prop:det_TUVW_with_zero_partition}.
\end{proof}

\begin{lem}\label{lem:sum_U_over_omega_low_degree}
Let $0\le s<n-1$. Then
\begin{equation}\label{eq:sum_U_over_omega_low_degree}
\sum_{j=1}^n\frac{\cU_s(z_j/2)}{\proddif_j(z)}=0.
\end{equation}
\end{lem}

\begin{proof}
Represent the left-hand side of~\eqref{eq:sum_U_over_omega_low_degree}
by~\eqref{eq:uquotient_particular}.
Since $0\le s<n-1$,
the first row of the determinant
coincides with one of the other rows,
and the determinant is zero.
\end{proof}

\begin{proof}[Proof of~\eqref{eq:palhom_as_sum_of_prod_U}]
Use \eqref{eq:Hpal} and \eqref{eq:U_generating}:
\[
\sum_{m=0}^\infty \homz_m(z)t^m
=\Hpal(z)(t)
=\prod_{j=1}^n \frac{1}{1-z_j t+t^2}
=\prod_{j=1}^n \sum_{q=0}^\infty \cU_q(z_j/2)t^q.
\]
Equating the coefficients of $t^m$ we arrive at
\eqref{eq:palhom_as_sum_of_prod_U}.
\end{proof}

\begin{proof}[Proof of~\eqref{eq:palhom_as_spz}]
Follows from~\eqref{eq:pal_Jacobi_Trudi_1}
and~\eqref{eq:JT1_spz} with $\la=(m)$.
The equivalence of~\eqref{eq:palhom_as_sum_of_U}
and~\eqref{eq:palhom_as_spz}
also follows from~\eqref{eq:spz_as_uquotient}
and Lemma~\ref{lem:uquotient_particular}.
\end{proof}

\begin{proof}[Proof of~\eqref{eq:palhom_via_oz}]
Identity~\eqref{eq:JT1_oz} for $\la=(s)$ yields
\[
\oz_{(s)}(z)=\homz_s(z)-\homz_{s-2}(z),
\]
i.e.\ $\homz_s(z)=\oz_{(s)}(z)+\homz_{s-2}(z)$.
By induction over $m$,
we obtain~\eqref{eq:palhom_via_oz}.
\end{proof}

\begin{proof}[Proof of~\eqref{eq:hom_via_palhom}]
By \eqref{eq:h_via_geometric_progressions} and \eqref{eq:monomial_as_lin_comb_U},
\[
\hom_m(z)
=\sum_{k=0}^{\lfloor(m+n-1)/2\rfloor}
\Catalan_{m+n-1-k,k}
\left(\sum_{j=1}^n
\frac{\cU_{m+n-1-2k}(z_j/2)}{\proddif_j(z)}\right).
\]
Lemma~\ref{lem:sum_U_over_omega_low_degree} allows
us to restrict the upper limit in the sum over $k$;
after that we apply~\eqref{eq:palhom_as_sum_of_U}:
\[
\hom_m(z)
=\sum_{k=0}^{\lfloor m/2\rfloor}
\Catalan_{m+n-1-k,k}
\left(\sum_{j=1}^n \frac{\cU_{m+n-1-2k}(z_j/2)}{\proddif_j(z)}\right)
=\sum_{k=0}^{\lfloor m/2\rfloor}
\Catalan_{m+n-1-k,k}\homz_{m-2k}(z).
\qedhere
\]
\end{proof}

Formula~\eqref{eq:subs_hpal} below
is inspired by~\cite[Lemma~1]{LitadaSilva2018}.
We abbreviate $(z_1,\ldots,z_n,z_{n+1})$ as $(z,z_{n+1})$.

\begin{prop}\label{prop:palhom_sum_dif}
For every $m$ in $\bN_0$,
\begin{align}
\label{eq:lower_degree_hpal}
&\homz_{m+2}(z,z_{n+1})
=\homz_{m+2}(z)
+z_{n+1}\homz_{m+1}(z,z_{n+1})
-\homz_m(z,z_{n+1}),
\\
\label{eq:subs_hpal}
&\homz_{m+1}(z,z_{n+1})-\homz_{m+1}(z,z_{n+2})
=(z_{n+1}-z_{n+2})\homz_m(z,z_{n+1},z_{n+2}).
\end{align}
\end{prop}

\begin{proof}
We start with a simple expression:
\[
\Epal(z,z_{n+1})(-t)=(1-z_{n+1}t+t^2)\Epal(z)(-t).
\]
Divide by the product $\Epal(z,z_{n+1})(-t)\Epal(z)(-t)$:
\[
\Hpal(z)(t)=(1-z_{n+1}t+t^2)\Hpal(z,z_{n+1})(t).
\]
By equating the coefficients of $t^{m+2}$ in both sides,
\[
\homz_{m+2}(z)
=\homz_{m+2}(z,z_{n+1})-z_{n+1}\homz_{m+1}(z,z_{n+1})
+ \homz_m(z,z_{n+1}).
\]
The obtained formula is equivalent to~\eqref{eq:lower_degree_hpal}.
Now consider the difference
\[
\Epal(z,z_{n+2})(-t)-\Epal(z,z_{n+1})(-t)
=(z_{n+1}-z_{n+2})t\Epal(z)(-t).
\]
Divide over the product $\Epal(z,z_{n+1})(-t)\Epal(z,z_{n+2})(-t)$:
\[
\Hpal(z,z_{n+1})(t)-\Hpal(z,z_{n+2})(t)= \frac{(z_{n+1}-z_{n+2})t\Hpal(z)(t)}{(1-z_{n+1}t+t^2)(1-z_{n+2}t+t^2)},
\]
i.e.
\[
\Hpal(z,z_{n+1})(t)-\Hpal(z,z_{n+2})(t)
=(z_{n+1}-z_{n+2})t\Hpal(z,z_{n+1},z_{n+2})(t).
\]
Equating the coefficients of $t^{m+1}$ we get~\eqref{eq:subs_hpal}.
\end{proof}

\subsection*{Formulas for $\boldsymbol{\powersumz}$}

In this subsection we prove Theorem~\ref{thm:pal_powersum}.
Formula~\eqref{eq:pal_powersum_as_sum_of_powersum} is another form of
\eqref{eq:powersum_sympl_via_powersum_z},
and~\eqref{eq:pal_powersum_via_ChebT_z}
is another form of~\eqref{eq:powersum_sympl_via_ChebT_z}.

\begin{proof}[Proof of~\eqref{eq:powersum_as_sum_palpowersum}]
Use~\eqref{eq:monomial_as_lin_comb_T}
and change the order of the sums:
\begin{align*}
\powersum_m(z)
&=\sum_{j=1}^nz_j^m
=\sum_{j=1}^n \sum_{k=0}^{\lfloor m/2\rfloor}
\coefmonomialviaT_{m,k}2\cT_{m-2k}\left(z_j/2\right)
=\sum_{k=0}^{\lfloor m/2\rfloor}\coefmonomialviaT_{m,k}
\sum_{j=1}^n\left(x_j^{m-2k}+x_j^{-(m-2k)}\right)
\\
&=\sum_{k=0}^{\lfloor m/2\rfloor}
\coefmonomialviaT_{m,k}\powersum_{m-2k}(x,x^{-1})
=\sum_{k=0}^{\lfloor m/2\rfloor}
\coefmonomialviaT_{m,k}\powersumz_{m-2k}(z).
\qedhere
\end{align*}
\end{proof}

Recall that the involution $\omega\colon\Sym\to\Sym$
can be defined by
\begin{equation}\label{eq:omega_powersum}
\omega(\powersum_m)
=(-1)^m\powersum_m\qquad (m\ge 1).
\end{equation}
In order to define $\omega_n\colon\Sym_n\to\Sym_n$,
we use the monomial embedding $\Sym_n\to\Sym$.
In other words, given $f$ in $\Sym_n$,
we expand $f$ in the monomial basis and then apply $\omega$.

\begin{prop}\label{prop:omega_palpowersum}
We have that
\begin{equation}\label{eq:omega_palpowersum}
\omega_n(\powersumz_m)(z)
=
\begin{cases}
(-1)^{m-1}\powersumz_m(z), & m\ \text{is odd},
\\[1ex]
(-1)^{m-1}\powersumz_m(z)
+\frac{1}{2}(-1)^{m/2}\binom{m}{m/2}\powersumz_0(z),
& m\ \text{is even}.
\end{cases}
\end{equation}
\end{prop}

\begin{proof}
Apply~\eqref{eq:pal_powersum_as_sum_of_powersum}
and~\eqref{eq:omega_powersum}.
For odd $m$,
\[
\omega_n(\powersumz_m)(z)
= \sum_{k=0}^{\lfloor m/2\rfloor}
\coefmonomialviaT_{m,k}(-1)^{m-2k-1}(-1)^k
\powersum_{m-2k}(z)
=(-1)^{m-1}\powersumz_m(z).
\]
For even $m$,
\[
\powersumz_m(z)
=\sum_{k=0}^{\frac{m}{2}-1}
\coefmonomialviaT_{m,k}(-1)^k\powersum_{m-2k}(z)
+\coefmonomialviaT_{m,m/2}(-1)^{m/2}
\powersum_0(z).
\]
Thus,
\begin{align*}
\omega_n(\powersumz_m)(z)
&=\sum_{k=0}^{\frac{m}{2}-1}
\coefmonomialviaT_{m,k}(-1)^{m-1}(-1)^k\powersum_{m-2k}(z)
+\coefmonomialviaT_{m,m/2}(-1)^{m/2}\powersum_0(z)
\\
&=(-1)^{m-1}\sum_{k=0}^{m/2}
\coefmonomialviaT_{m,k}(-1)^{m-1}(-1)^k\powersum_{m-2k}(z)
+2\coefmonomialviaT_{m,m/2}(-1)^{m/2}\powersum_0(z).
\end{align*}
The last expression yields
the second case of~\eqref{eq:omega_palpowersum}.
\end{proof}

Unfortunately, in general $\omega_n$
does not convert $\homz^{(n)}_\la$ into $\elemz^{(n)}_\la$,
even with the restriction $\ell(\la)\le n$.

\subsection*{Minimal generating subsets}

Here we consider $\Sym_n$
as a complex algebra with identity.
Given a subset $S$ of $\Sym_n$,
we denote by $\langle S\rangle$
the subalgebra (with identity) generated by $S$.
It is well known that $\{\elem_m\}_{m=1}^n$
is a minimal generating subset of $\Sym_n$.
The following elementary lemma, which we will not prove,
is useful to find other minimal generating subsets of $\Sym_n$.

\begin{lem}\label{lem:minimal_generating_subsets}
Let  $a_1,\ldots,a_n,b_1,\ldots,b_n\in\Sym_n$
such that
\begin{itemize}
\item[1)] $\deg(a_m)=m$ for every $m$ in $\{1,\ldots,n\}$,
\item[2)] $\{a_1,\ldots,a_n\}$
is a minimal generating subset of
$\Sym_n$,
\item[3)] $b_m\in\langle a_1,\ldots,a_m\rangle$
for every $m$ in $\{1,\ldots,n\}$,
\item[4)] $a_m\in\langle b_1,\ldots,b_m\rangle$
for every $m$ in $\{1,\ldots,n\}$.
\end{itemize}
Then $\{b_1,\ldots,b_n\}$
is also a minimal generating subset of
$\Sym_n$.
\end{lem}

\begin{proof}[Proof of Theorem~\ref{thm:minimal_generating_subsets}]
Lemma~\ref{lem:minimal_generating_subsets}
and the formula $\Elem(z)(-t)\Hom(z)(t)=1$
imply that $\{\hom_m\}_{m=1}^n$
is a minimial generating subset of $\Sym_n$.
Newton's identities yield a similar conclusion
for $\{\powersum_m\}_{m=1}^n$.
Now Theorem~\ref{thm:minimal_generating_subsets} follows from
Lemma~\ref{lem:minimal_generating_subsets}
and the identities
\eqref{eq:elemz_via_elem},
\eqref{eq:elem_via_elemz},
\eqref{eq:palhom_via_hom},
\eqref{eq:hom_via_palhom},
\eqref{eq:pal_powersum_as_sum_of_powersum},
and \eqref{eq:powersum_as_sum_palpowersum}.
\end{proof}

\subsection*{\texorpdfstring{Formulas for
$\boldsymbol{\elemz^{\odd}}$,
$\boldsymbol{\homz^{\odd}}$, and
$\boldsymbol{\powersumz^{\odd}}$}%
{Formulas for elemzodd, palhomodd, palpowersumodd}}

With the generating series for the sequences
$(\elem_m)_{m=0}^\infty$
and $(\hom_m)_{m=0}^\infty$ it is easy to
obtain the following recurrence relations:
\begin{align}
\label{eq:elem_recur}
\elem_{m+1}(y_1,\ldots,y_p,y_{p+1})
&=\elem_{m+1}(y_1,\ldots,y_p)+y_{p+1}\elem_m(y_1,\ldots,y_p),
\\
\label{eq:hom_recur}
\hom_{m+1}(y_1,\ldots,y_p,y_{p+1})
&=\hom_{m+1}(y_1,\ldots,y_p) + y_{p+1}\hom_m(y_1,\ldots,y_p,y_{p+1}).
\end{align}
The definition of $\powersum_m$ immediately yields
\begin{equation}\label{eq:powersum_recur}
\powersum_m(y_1,\ldots,y_p,y_{p+1})
=\powersum_m(y_1,\ldots,y_p)+y_{p+1}^m.
\end{equation}
Apply these formulas to the odd symplectic alphabet $x,x^{-1},1$,
then use the morphisms $\Phi_n$ and $\Phi^{\odd}_n$:
\begin{align}
\label{eq:elem_symplectic_odd_via_even}
\elemz^{\odd}_{m+1}(z)
&=\elemz_{m+1}(z)+\elemz_m(z),
\\
\label{eq:hom_symplectic_odd_via_even}
\homz^{\odd}_{m+1}(z)&=\homz_{m+1}(z)+\homz^{\odd}_m(z),
\\
\label{eq:powersum_symplectic_odd_via_even}
\powersumz^{\odd}_m(z)&=\powersumz_m(z)+1.
\end{align}
These identities and Theorems~\ref{thm:pal_elem}, \ref{thm:pal_hom}, \ref{thm:pal_powersum} yield the following formulas
for $\elemz^{\odd}_m(z)$, $\homz^{\odd}_m(z)$, and $\powersumz^{\odd}_m(z)$.

\begin{prop}\label{prop:elemz_odd}
For every $m$ in $\{0,\ldots,2n+1\}$,
\begin{align}
\label{eq:elemz_odd_via_elem}
\elemz^{\odd}_m(z)
=\sum_{k=\max \lbrace 0,2m-2n-1\rbrace}^{m}\binom{n-m+k}{\lfloor k/2 \rfloor} \elem_{m-k}(z).
\end{align}
\end{prop}

\begin{prop}\label{prop:palhom_odd}
For every $m$ in $\bN_0$,
\begin{align}
\label{eq:palhom_odd_via_palhom}
\homz^{\odd}_m(z)
&=\sum_{k=0}^{m}\homz_{k}(z),
\\
\label{eq:palhom_odd_via_spz}
\homz^{\odd}_m(z)
&=\sum_{k=0}^m \spz_{(k)}(z),
\\
\label{eq:palhom_odd_via_quotient_sum_U}
\homz^{\odd}_m(z)
&=\sum_{j=1}^n \frac{\cU^{(1)}_{m+n-1}(z_j/2)}{\proddif_j(z)},
\\
\label{eq:palhom_odd_via_hom}
\homz^{\odd}_m(z)
&=\sum_{k=0}^m \coefsumU_{m+n-1,k+n-1}\hom_{k}(z),
\end{align}
where $\coefsumU_{m,k}$ is defined by~\eqref{eq:coefsumU}.
\end{prop}

\begin{prop}\label{prop:palpowersum_odd}
We have
\begin{align}
\label{eq:palpowersum_odd_via_powersum}
\powersumz^{\odd}_m(z)
= 
\begin{cases}
\displaystyle 1+m\sum_{j=0}^{\lfloor m/2 \rfloor}
\frac{(-1)^j}{m-j}\binom{m-j}{j}\powersum_{m-2j}(z_1,\ldots,z_n), & m\in\bN;\\[0.5ex]
1+2\powersum_0(z_1,\ldots,z_n), & m=0.
\end{cases}
\end{align}
\end{prop}

\section{Schur polynomials in symplectic variables}
\label{sec:palschur}

We denote by $\schur_{\la/\mu}$ the skew Schur function
associated to a skew partition $\la/\mu$.
It is given by the following Jacobi--Trudi formulas:
\[
\schur_{\la/\mu}
\eqdef
\det\bigl[\hom_{\la_j-\mu_k-j+k}\bigr]_{j,k=1}^{\ell(\la)},
\qquad
\schur_{\la/\mu}
=\det\bigl[\elem_{\la^{\prime}_j-\mu^{\prime}_k-j+k}\bigr]_{j,k=1}^{\la_1}.
\]
Put $\schurz_{\la/\mu}^{(n)}\eqdef\Phi_n(\schur_{\la/\mu}^{(2n)})$.
Since $\Phi_n$ is a homomorphism of algebras,
the polynomials $\schurz_{\la/\mu}^{(n)}$ inherit
some properties of the classical Schur polynomials.
In particular, $\schurz_{\la/\mu}^{(n)}$ can be computed by the
following analogs of the Jacobi--Trudi formulas:
\begin{align}
\label{eq:pal_Jacobi_Trudi_1}
\schurz_{\la/\mu}(z)
&=\det\bigl[\homz_{\la_j-\mu_k-j+k}(z)
\bigr]_{j,k=1}^{\ell(\la)},
\\
\label{eq:pal_Jacobi_Trudi_2}
\schurz_{\la/\mu}(z)
&=\det\bigl[\elemz_{\la^{\prime}_j-\mu^{\prime}_k-j+k}(z)
\bigr]_{j,k=1}^{\la_1}.
\end{align}
Furthermore, 
\[
\schurz_\la(z) \schurz_\mu(z)
=\sum_{\nu} \LR_{\la,\mu}^\nu \schurz_\nu(z),
\]
where $\LR_{\la,\mu}^\nu$ are the usual
Littlewood--Richardson coefficients for the Schur functions.

In this section we prove Theorems~\ref{thm:basis}
and \ref{thm:Cauchy_palindromic},
and state a few other properties
of the polynomials $\schurz_\la^{(n)}$.

\begin{prop}\label{prop:palschur_zero}
Let $\la\in\Partitions_m$ with $m>2n$.
Then $\schurz_\la(z_1,\ldots,z_n)=0$.
\end{prop}

\begin{proof}
Since $\ell(\la)>2n$, the Jacobi--Trudi formula easily implies that
\[
\schur_\la(y_1,\ldots,y_{2n})=0.
\]
Now $\schurz_\la(z_1,\ldots,z_n)=0$
by definition of the morphism $\Phi_n$.
\end{proof}

The following ``symmetry property'' of $\schurz$
is equivalent to~\cite[Lemma~1]{Krattenthaler1998}
and follows from
the dual Jacobi--Trudi identity~\eqref{eq:pal_Jacobi_Trudi_2}
and Proposition~\ref{prop:elemz_symmetry}.

\begin{prop}\label{prop:palschur_symmetry}
For $1\leq k\leq n$
\begin{equation*}
\schurz_{(\la_1,\ldots,\la_{n+k})}(z_1,\ldots,z_n)
=\schurz_{\mu}(z_1,\ldots,z_n),
\end{equation*}
where $\mu=(\la_1^{n-k},\la_1-\la_{n+k},\la_1- \la_{n+k-1},\ldots,\la_1-\la_2).$
\end{prop}

To prove Theorem~\ref{thm:basis},
we need an elementary lemma from linear algebra
that can be proved by induction.

\begin{lem}\label{lem:basis}
Let $V$ be a vector space,
$(b_\la)_{\la\in J}$ be a basis of $V$,
and the index set $J$ be the union of the sequence
$(J_w)_{w=0}^\infty$ of finite sets $J_w$,
such that $J_w\subseteq J_{w+1}$
for every $w$ in $\bN_0$.
Put $J_{-1}\eqdef\varnothing$.
Denote by $V_w$ the subspace generated by $b_\la$
with $\la$ in $J_w$.
Suppose that $(a_\la)_{\la\in J}$ is a family
of vectors in $V$ such that for every $w$ in $\bN_0$
and every $\la$ in $J_w\setminus J_{w-1}$,
\[
a_\la - b_\la \in V_{w-1}.
\]
Then $(a_\la)_{\la\in J}$ is a basis of $V$.
\end{lem}

\begin{proof}[Proof of Theorem~\ref{thm:basis}]
For every $w$ in $\bN_0$,
put $J_w\eqdef\{\la\in\Partitions_m\colon\ |\la|\le w\}$
and denote by $V_w$ the
subspace generated by $\{\schur_\la^{(n)}\colon\ \la\in J_w\}$.
It is well known that the family
$(\schur_\la^{(n)})_{\la\in\Partitions_n}$
is a basis of the vector space $\Sym_n$,
and $V_w$ consists of all symmetric polynomials
in $n$ variables of degree $\le w$.

Given a partition $\la$ in $\Partitions_n$,
write $\schurz_\la(z)$ in the form~\eqref{eq:pal_Jacobi_Trudi_1}
and expand $\homz_{\la_j-j+k}(z)$ into a linear combination
of $\hom_m(z)$ with $m\le\la_j-j+k$,
using~\eqref{eq:palhom_via_hom}.
Thereby it can be shown that
\[
\schurz_\la(z)=\schur_\la(z)+R(z),
\]
where $R$ is a symmetric polynomial of degree
strictly less than $|\la|$ and with integer coefficients.
So, the conditions of Lemma~\ref{lem:basis} are fulfilled,
and $(\schurz^{(n)}_\la)_{\la\in\Partitions_n}$
is a basis of the vector space $\Sym_n$.

For the other families from~Theorem~\ref{thm:basis},
the proofs are similar.
For the families $(\spz^{(n)}_\la)_{\la\in\Partitions_n}$,
$(\oz^{(n)}_\la)_{\la\in\Partitions_n}$,
the bialternant identities can be used
instead of the Jacobi--Trudi formulas.
\end{proof}

\begin{proof}[Proof of Theorem~\ref{thm:Cauchy_palindromic}]
All formulas in Theorem~\ref{thm:Cauchy_palindromic}
follow easily from the classical Cauchy and dual Cauchy identities.
Let us verify only~\eqref{eq:Cauchy_palindromic}.
Representing $z_j$ as $x_j+x_j^{-1}$,
\begin{align*}
\sum_{\la\in\Partitions}\schurz_\la(z)\schur_\la(y)
&=\sum_{\la\in\Partitions}\schur_\la(x,x^{-1})\schur_\la(y)
\\
&=\prod_{j=1}^n\prod_{k=1}^m \frac{1}{(1-x_j y_k)(1-x_j^{-1} y_k)}
=\prod_{j=1}^n\prod_{k=1}^m \frac{1}{1-z_j y_k+y_k^2}.
\qedhere
\end{align*}
\end{proof}

\begin{rem}
The sum in the left-hand side of~\eqref{eq:Cauchy_palindromic}
can be reduced to the partitions $\la$ with
$\ell(\la) \leq \min\{2n,m\}$,
and the sum in the left-hand side of~\eqref{eq:Cauchy_palindromic_odd}
can be reduced to the partitions $\la$
with $\ell(\la) \leq \min \{2n+1,m\}$.
These sums can be treated in the formal sense.
They are absolutely converging for small values of $z_j$ and $y_k$.
The sums in the left-hand sides of~\eqref{eq:dual_Cauchy_palindromic},
\eqref{eq:dual_Cauchy_schurz_schurz},
\eqref{eq:dual_Cauchy_palindromic_odd},
and~\eqref{eq:dual_Cauchy_schurzodd_schurzodd}
are finite.
For example, the sum in the left-hand side of~\eqref{eq:dual_Cauchy_palindromic}
can be reduced to the partitions $\la$ satisfying
$\ell(\la)\leq 2n$ and $\la_1\le m$.
\end{rem}

There are simple bialternant formulas
for $\schurz_{(m)}^{(n)}$ and $\schurz_{(1^m)}^{(n)}$,
see~\eqref{eq:elemz_as_det} and~\eqref{eq:homz_as_det}.
The next proposition shows that there is no similar formula
for $\schurz_\la^{(n)}$ with general $\la$.

\begin{prop}\label{prop:no_bialternant_schurz}
There do not exist univariate polynomials
$f$ and $g$ such that
\begin{equation}\label{eq:palschur21_bialternant}
\schurz_{(2,1)}(z_1,z_2)
=\frac{\detmatr{cc}{ f(z_1) & f(z_2) \\ g(z_1) & g(z_2)}}%
{\VandermondePol(z_1,z_2)}.
\end{equation}
\end{prop}

\begin{proof}
Using~\eqref{eq:pal_Jacobi_Trudi_1} we obtain
\begin{equation}\label{eq:palschur21_explicit}
\schurz_{(2,1)}(z_1,z_2)
=z_1^2 z_2 + z_1 z_2^2 + z_1 + z_2.
\end{equation}
Suppose that $f$ and $g$ are univariate polynomials
satisfying~\eqref{eq:palschur21_bialternant}, and
\[
f(t)=\sum_{j=0}^p f_j t^j,\qquad
g(t)=\sum_{j=0}^q g_j t^j,
\]
with $f_p\ne0$, $g_q\ne0$.
If $p=q$, then an elementary operation with the rows
of the determinant in the numerator
allows to pass to the case $q=p-1$.
So, without loss of generality, we consider the case $p>q$.
Then the leading terms of the numerator
of~\eqref{eq:palschur21_bialternant}
are $f_p g_p z_1^p z_2^q$ and $-f_p g_p z_1^q z_2^p$,
and it is easy to conclude that $p=3$, $q=1$.
Since $f_p g_p=1$, we may consider the case $f_p=1$ and $g_p=1$.
An elementary computation yields
\[
\frac{\detmatr{cc}{ f(z_1) & f(z_2) \\ g(z_1) & g(z_2)}}{z_1 - z_2}
=z_1^2 z_2 + z_1 z_2^2 + g_0(z_1^2 + z_2^2) + (g_0 + f_2)z_1 z_2 + g_0 f_2 (z_1 + z_2) + g_0 f_1 - f_0.
\]
From this and~\eqref{eq:palschur21_explicit},
$g_0 = 0$ and simultaneously $g_0 f_2=1$,
which is impossible.
\end{proof}

However, there is a bialternant formula for $\schurz_{(2,1)}(z_1,z_2)$
with a more complicated denominator.

\begin{example}\label{example:bialternant_schurz_21}
It can be verified directly that
\begin{equation}\label{eq:bialternant_schurz_21}
\schurz_{(2,1)}(z_1,z_2)
=\frac{\detmatr{cc}{
z_1^4+1 & z_2^4 + 1 \\
z_1^2 & z_2^2}}%
{\detmatr{cc}{
z_1^2+1 & z_2^2+1 \\
z_1 & z_2
}}.
\end{equation}
\end{example}

\begin{rem}\label{rem:schurz_bialternant_general}
A natural problem for future work
is to generalize Example~\ref{example:bialternant_schurz_21},
i.e. for every $n$ in $\bN$ and $\la$ in $\Partitions_n$
find univariate polynomials $f_1,\ldots,f_n$, $g_1,\ldots,g_n$,
with coefficients depending on $\la$ and $n$, such that
\[
\schurz_\la(z_1,\ldots,z_n)
=\frac{\det\bigl[f_j(z_k)\bigr]_{j,k=1}^n}%
{\det\bigl[g_j(z_k)\bigr]_{j,k=1}^n}.
\]
\end{rem}

\begin{tempremark}
Unfortunately, we are unable to solve this problem at the moment.
We tried to solve this problem for~$\schurz_{(3,1)}(z_1,z_2)$.
A simple reasoning for the leader terms shows
that the degrees of $f_1,f_2,g_1,g_2$
should be of the form $s+4,s+1,s+1,s$.
We have proved that there is no solution for $s=0$
(a simple reasoning with the several leader terms)
nor for $s=1$ (the proof is boring and involves many terms).
\end{tempremark}

Littlewood~\cite[Appendix]{Littlewood1940} expanded $\schur_\la$
into linear combinations of $\sympl_\mu$ or $\orth_\mu$:
\begin{equation}\label{eq:schur_via_sympl_or_orth}
\schur_\la
=\sum_{\mu\in\Partitions}
\left(\sum_{\substack{\nu\in\Partitions\\
\nu'\ \text{even}}}\LR_{\nu,\mu}^\la\right)
\sympl_\mu,
\qquad
\schur_\la
=\sum_{\mu\in\Partitions}
\left(\sum_{\substack{\nu\in\Partitions\\
\nu\ \text{even}}}\LR_{\nu,\mu}^\la\right)\orth_\mu.
\end{equation}
Here $\LR$ are the Littlewood--Richardson coefficients,
and the restrictions ``$\nu$ even'' or ``$\nu'$ even''
mean that all parts of $\nu$ or $\nu'$, respectively, are even.
Evaluating both sides of the identities~\eqref{eq:schur_via_sympl_or_orth}
at the symplectic list of variables
and applying the homomorphism $\Phi_n$,
yields expansions of $\schurz_\la^{(n)}$
into linear combinations of $\spz_\mu^{(n)}$ or $\oz_\mu^{(n)}$,
with coefficients not depending on $n$.
Krattenthaler~\cite{Krattenthaler1998} studied
particular cases of~\eqref{eq:schur_via_sympl_or_orth}
for partitions $\la$ of ``nearly rectangular form''.
In that cases the coefficients are always $0$ or $1$.

\section{\texorpdfstring{Connection to the determinants and minors\\
of banded symmetric Toeplitz matrices}{Connection to the determinants and minors
of banded symmetric Toeplitz matrices}}
\label{sec:symm_Toeplitz_minors}

Recall that Toeplitz matrices are of the form
$T_n(a)=[a_{j-k}]_{j,k=1}^n$,
where $a_j$ are some entries.
Various equivalents formulas for banded Toeplitz determinants
were found by Baxter and Schmidt~\cite{BaxterSchmidt1961},
and Trench~\cite{Trench1985eig}.
Bump and Diaconis~\cite{BumpDiaconis2002},
Lascoux~\cite{Lascoux2003},
and other authors noticed relations
between Toeplitz minors and skew Schur polynomials.
It was shown explicitly in \cite{A2012,MM2017}
that every minor of the $m\times m$ Toeplitz matrix $T_m(a)$
generated by a Laurent polynomial $a$,
can be written as a certain skew Schur polynomial
evaluated at the roots of $a$.

Many applications and investigations
involve Hermitian Toeplitz matrices
\cite{GrenanderSzego1958,BG2005,BGM2010,BBGM2015}.
In particular, an important object of study
are symmetric banded Toeplitz matrices
generated by palindromic Laurent polynomials
\begin{equation}\label{eq:Laurent_palindromic}
a(t)
=\sum_{k=-n}^n a_k t^k
=a_0+\sum_{k=1}^n a_k (t^k + t^{-k}),
\end{equation}
where $a_n\ne0$ and $a_k=a_{-k}$ for all $k$.
For $|k|>n$, we put $a_k=0$.

In this section we consider minors and determinants
of Toeplitz matrices $T_n(a)$ generated by such polynomials.

By Proposition~\ref{prop:palindromic_pol_through_ChebT},
there exists a polynomial $g$ of degree $n$ such that
$a(t)=g(t+1/t)$, and we denote by $z_1,\ldots,z_n$ the zeros of $g$:
\begin{equation}\label{eq:g_from_palindromic_Laurent_polynomial}
g(u)
=a_n+\sum_{k=1}^n a_{n-k} 2\cT_k(u/2)
=a_n \prod_{j=1}^n (u-z_j).
\end{equation}
First, we notice that all minors of symmetric Toeplitz matrices
can be expressed through $\schurz_{\la/\mu}$
with a certain skew partition $\la/\mu$.
We write $\id_d$ for $(1,\ldots,d)$ and
$\rev(\al_1,\ldots,\al_d)$ for $(\al_d,\ldots,\al_1)$.

\begin{prop}\label{prop:minor_Toeplitz_symmetric}
Let $a$ be a palindromic Laurent polynomial of the form~\eqref{eq:Laurent_palindromic},
and $z_1,\ldots,z_n$ are the zeros of the polynomial $g$
defined by \eqref{eq:g_from_palindromic_Laurent_polynomial}.
Furthermore, let $m\in\bN$, $r\le m$,
$\rho_1,\ldots,\rho_r\in\{1,\ldots,m\}$,
$\sigma_1,\ldots,\sigma_r\in\{1,\ldots,m\}$,
such that $\rho_1<\dotsb<\rho_r$ and $\sigma_1<\dotsb<\sigma_r$.
Then
\begin{align}
\label{eq:minor_Toeplitz_symmetric1}
\det T_m(a)_{\rho,\sigma}
&=(-1)^{rn+|\rho|+|\sigma|}a_n^r
\schurz_{(r^n,\rev(\xi-\id_d))/(\rev(\eta-\id_d))}(z_1,\ldots,z_n)
\\
\label{eq:minor_Toeplitz_symmetric2}
\det T_m(a)_{\rho,\sigma}
&=(-1)^{rn+|\rho|+|\sigma|}a_n^r
\schurz_{(r^n,r^d+\id_d-\eta)/(r^d+\id_d-\xi)}(z_1,\ldots,z_n),
\end{align}
where $d=m-r$,
$\{\xi_1,\ldots,\xi_d\}
=\{1,\ldots,m\}\setminus\{\rho_1,\ldots,\rho_r\}$,
$\{\eta_1,\ldots,\eta_d\}
=\{1,\ldots,m\}\setminus\{\sigma_1,\ldots,\sigma_r\}$,
$\xi_1<\dotsb<\xi_d$, $\eta_1<\dotsb<\eta_d$.
\end{prop}

\begin{proof}
According to~\cite{MM2017},
for a general Laurent polynomial of the form
\[
a(t)=\sum_{k=-q}^p a_k t^k
=a_p t^{-q}\prod_{j=1}^{p+q} (t-x_j),
\]
the minor $\det T_m(a)_{\rho,\sigma}$ can be expressed
as the following skew Schur polynomial
in the variables $x_1,\ldots,x_{p+q}$:
\begin{align}
\label{eq:minor_Toeplitz_general1}
\det T_m(a)_{\rho,\sigma}
&=(-1)^{nr+|\rho|+|\sigma|}a_p^r
\schur_{(r^n,\rev(\xi-\id_d))/(\rev(\eta-\id_d))}(x_1,\ldots,x_{p+q})
\\
\label{eq:minor_Toeplitz_general2}
&=(-1)^{nr+|\sigma|+|\rho|}a_p^r
\schur_{(r^n,r^d+\id_d-\eta)/(r^d+\id_d-\xi)}(x_1,\ldots,x_{p+q}).
\end{align}
In the palindromic case~\eqref{eq:Laurent_palindromic},
we use notation $\schurz$ and obtain the desired formula.
\end{proof}

In the three corollaries below, let $a$ and $z_1,\ldots,z_n$
be as in Proposition~\ref{prop:minor_Toeplitz_symmetric},
and $m\in\bN$.

\begin{cor}\label{cor:det_Toeplitz_symmetric}
We have that
\[
\det T_m(a)=(-1)^{nm} a_n^m\,\schurz_{(m^n)}(z).
\]
\end{cor}

\begin{cor}\label{cor:cofactor_Toeplitz_symmetric}
Let $p,q\in\{1,\ldots,m\}$.
Then the $(p,q)$th entry of the adjugate matrix of $T_m(a)$ is
\[
(\adjugate(T_m(a)))_{p,q}= (-1)^{n(m-1)}a_n^{m-1}\schurz_{((m-1)^n,q-1)/(p-1)}(z).
\]
\end{cor}

\begin{cor}\label{cor:eigenvector_Toeplitz_symmetric}
Suppose that $\det T_m(a)=0$.
Then the vector $v=[v_q]_{q=1}^m$ with components
\[
v_q=\schurz_{((m-1)^{n-1},m-q)}(z)
\]
belongs to the nullspace of $T_m(a)$.
\end{cor}

\begin{rem}
Since the Toeplitz matrices are persymmetric,
it is easy to select two different submatrices
of a large Toeplitz matrix such that
their determinants coincide.
So, for fixed $n$ and $m$, with $m$ large enough,
the correspondence $(\rho,\sigma)\mapsto \schurz_{\la/\mu}^{(n)}$,
defined in Proposition~\ref{prop:minor_Toeplitz_symmetric},
is not injective.
\end{rem}

Furthermore, we will show that every polynomial of the form
$\schurz_{\la/\mu}(z)$
can be written as a minor of a banded symmetric Toeplitz matrix.
Notice that the skew partition $\la/\mu$
in~\eqref{eq:minor_Toeplitz_symmetric1}
has a special form: the initial entries of $\la$ coincide.
The main idea is canceling them out
with the initial entries of $\mu$.

\begin{prop}\label{prop:minor_Toeplitz_symmetric_equal_to_given_skew_Schur}
Let $\la,\mu\in\Partitions_n$, $\mu\le\la$,
let $a$ be the palindromic Laurent polynomial given by
\[
a(t)=\prod_{j=1}^n (t-z_j+t^{-1}),
\]
and let $m\in\bN$ with $m>n+\ell(\la)$.
Then there exist $r\le m$,
$\rho_1,\ldots,\rho_r\in\{1,\ldots,m\}$,
$\sigma_1,\ldots,\sigma_r\in\{1,\ldots,m\}$,
such that $\rho_1<\dotsb<\rho_r$, $\sigma_1<\dotsb<\sigma_r$, and
\begin{equation}
\label{eq:minor_Toeplitz_symmetric_equal_to_given_skew_Schur}
\det T_m(a)_{\rho,\sigma}
=(-1)^{rn+|\rho|+|\sigma|}
\schurz_{\la/\mu}(z_1,\ldots,z_n).
\end{equation}
\end{prop}

\begin{proof}
Put $q\eqdef\ell(\la)$, $d\eqdef n+q$, $r\eqdef m-d$,
\begin{equation}\label{eq:xi_eta_via_la_mu}
\xi\eqdef(\id_n,n^q+\id_q+\rev(\la)),\qquad
\eta\eqdef(\id_q+\rev(\mu),(m-n-q)^n+\id_n),
\end{equation}
i.e.
\begin{align*}
&\xi_1\eqdef 1,\ \ldots,\ \xi_n\eqdef n,
&
&\xi_{n+1}\eqdef n+1+\la_q,\ \ldots,\ \xi_{n+q}=n+q+\la_1,
\\
&\eta_1\eqdef 1+\mu_q,\ \ldots,\ \eta_q\eqdef q+\mu_1,
&
&\eta_{q+1}\eqdef m-n-q+1,\ \ldots,\ \eta_{q+n}\eqdef m-q.
\end{align*}
Moreover, define $\rho_1,\ldots,\rho_r$ to be the elements
of $\{1,\ldots,m\}\setminus\{\xi_1,\ldots,\xi_d\}$,
enumerated in the ascending order,
and $\sigma_1,\ldots,\sigma_r$ to be the elements
of $\{1,\ldots,m\}\setminus\{\eta_1,\ldots,\eta_d\}$
enumerated in the ascending order.
Then
\[
(r^n,\rev(\xi-\id_d))/(\rev(\eta-\id_d))
=((m-n-q)^n,\la,0^n)/((m-n-q)^n,\mu),
\]
and by Proposition~\ref{prop:minor_Toeplitz_symmetric}
we obtain~\eqref{eq:minor_Toeplitz_symmetric_equal_to_given_skew_Schur}.
\end{proof}

Corollary~\ref{cor:det_Toeplitz_symmetric}
relates banded symmetric Toeplitz determinants
with Schur polynomials corresponding to rectangular partitions
and evaluated at the symplectic alphabet.
Efficient formulas for these polynomials
were independently found
in~\cite{Trench1987symm,Krattenthaler1998,Elouafi2014}.
As before, write $x$, $x^{-1}$, and $z$ instead of $x_1,\ldots,x_n$, $x_1^{-1},\ldots,x_n^{-1}$,
and $z_1,\ldots,z_n$, respectively.

\begin{prop}\label{prop:palschur_rectangular_decomposition}
We have
\begin{equation}\label{eq:Krattenthaler_factorization}
\schur_{(m^n)}(x,x^{-1})
=
\begin{cases}
\sympl_{((p-1)^n)}(x,x^{-1})\orth_{(p^n)}(x,x^{-1}), & m=2p-1,\\
\orth_{(p^n)}(x,x^{-1},1)\orth_{(p^n)}(x,x^{-1},-1), & m=2p,
\end{cases}
\end{equation}
i.e.
\begin{equation}\label{eq:Elouafi_factorization}
\schurz_{(m^n)}(z)
=
\begin{cases}
\spz_{((p-1)^n)}(z)\oz_{(p^n)}(z), & m=2p-1,\\
(-1)^{pn}\oz^{\odd}_{(p^n)}(z)\oz^{\odd}_{(p^n)}(-z), & m=2p.
\end{cases}
\end{equation}
Equivalently,
\begin{equation}\label{eq:Trench_factorization}
\schurz_{(m^n)}(u_1^2-2,\ldots,u_n^2-2)
=
\frac{ \det \bigl[ 2\cT_{m+2j-1}(u_k/2) \bigr]_{j,k=1}^n\
\det \bigl[ \cU_{m+2j-2}(u_k/2) \bigr]_{j,k=1}^n}%
{ \det \bigl[ 2\cT_{2j-1}(u_k/2) \bigr]_{j,k=1}^n\ 
\det \bigl[ \cU_{2j-2}(u_k/2) \bigr]_{j,k=1}^n}.
\end{equation}
\end{prop}

The equivalence between~\eqref{eq:Elouafi_factorization}
and~\eqref{eq:Trench_factorization}
follows from Proposition~\ref{prop:spz_oz_duplication}.

Comparing to~\eqref{eq:Krattenthaler_factorization}
or~\eqref{eq:Elouafi_factorization},
formula~\eqref{eq:Trench_factorization}
has a ``defect'' that it uses auxiliary variables $u_1,\ldots,u_n$,
related with $z_1,\ldots,z_n$ by $z_j=u_j^2-2$.
On the other hand, 
this defect is not important after applying
trigonometric or hyperbolic changes of variables
(say, $u_j=\cos(\theta_j/2)$ and $z_j=\cos(\theta_j)$),
and formula~\eqref{eq:Trench_factorization}
can be more convenient in some applications
because it joins two cases appearing
in~\eqref{eq:Krattenthaler_factorization}
and~\eqref{eq:Elouafi_factorization}.

Ciucu and Krattenthaler~\cite{CiucuKrattenthaler2009}
proved~\eqref{eq:Krattenthaler_factorization}.
Trench and Elouafi did not use
the language of symmetric polynomials.
Trench~\cite{Trench1987symm}
worked with symmetric Toeplitz matrices generated by rational functions.
For the case of symmetric Toeplitz matrices,
his result is equivalent to~\eqref{eq:Trench_factorization},
after trigonometric changes of variables.
Elouafi~\cite{Elouafi2014} worked with
banded symmetric Toeplitz determinants;
his result is equivalent to~\eqref{eq:Elouafi_factorization},
with the right-hand side written in the bialternant form,
see Theorem~\ref{thm:spz_oz_bialternant}.

Recently Ayyer and Behrend~\cite[formulas (18) and (19)]{AyyerBehrend2019} generalized~\eqref{eq:Krattenthaler_factorization}
to partitions of the symmetric form
\[
(2\la_1,\la_1+\la_2,\ldots,\la_1+\la_n,
\la_1-\la_n,\ldots,\la_1-\la_2,\la_1-\la_1)
\]
or
\[
(2\la_1+1,\la_1+\la_2+1,\ldots,\la_1+\la_n+1,
\la_1-\la_n,\ldots,\la_1-\la_2,\la_1-\la_1).
\]
Notice that the proof in~\cite{AyyerBehrend2019}
uses exactly the same ideas as the proof in~\cite{Trench1987symm}:
in both cases one starts
with the bialternant formula for $\schur_{(m^n)}(x,x^{-1})$,
then applies elementary transformations of determinants,
and reduces them to block-triangular form.

For readers' convenience,
we explain below the idea of the proof given by Elouafi,
but in the language of symmetric polynomials.

\begin{proof}[Proof of~\eqref{eq:Elouafi_factorization}
following~\cite{Elouafi2014}]
Start with the Jacobi--Trudi formula for $\schurz_{(m^n)}(z)$:
\begin{equation}\label{eq:palschur_Elouafi_Jacobi_Trudi}
\schurz_{(m^n)}(z)
=\det\bigl[\homz_{m-j+k}(z)\bigr]_{j,k=1}^n.
\end{equation}
Using~\eqref{eq:palhom_as_spz} and Lemma~\ref{lem:sum_U_over_omega_low_degree},
it is possible to derive the following expansions of $\homz$:
\begin{align}
\label{eq:hpal_Elouafi_expansion_TU}
\homz_{2p+1-j+k}(z)
&=\sum_{s=1}^n \frac{2\cT_{n+p+1-j}(z_s/2)\,\cU_{p+k-1}(z_s/2)}%
{\proddif_s(z)},
\\
\label{eq:hpal_Elouafi_expansion_VW}
\homz_{2p-j+k}(z)
&=\sum_{s=1}^n \frac{\cW_{n+p-j}(z_s/2)\,\cV_{p+k-1}(z_s/2)}%
{\proddif_s(z)}.
\end{align}
Applying~\eqref{eq:hpal_Elouafi_expansion_TU} in the case $m=2p+1$
or \eqref{eq:hpal_Elouafi_expansion_VW} in the case $m=2p$,
one can write the determinant in the right-hand side
of~\eqref{eq:palschur_Elouafi_Jacobi_Trudi}
as a product of determinants.
\end{proof}

\subsection*{Acknowledgements}

The contribution of the first author has been funded by 
the Swedish Research Council (Vetenskapsr{\r{a}}det), grant 2015-05308.
The contribution of the second, third, and fourth authors
has been supported by IPN-SIP projects
(Instituto Polit\'{e}cnico Nacional, Mexico)
and CONACYT (Mexico).

\clearpage

\bigskip\noindent
Per Alexandersson\newline
\myurl{https://orcid.org/0000-0003-2176-0554}\newline
e-mail: per.w.alexandersson@gmail.com

\bigskip\noindent
Dept. of Mathematics\newline
Stockholm University\newline
SE-10691 Stockholm\newline
Sweden

\bigskip\bigskip\noindent
Luis Angel Gonz\'{a}lez-Serrano\newline
\myurl{https://orcid.org/0000-0001-6330-1781}\newline
e-mail: lags1015@gmail.com

\medskip\noindent
Egor A. Maximenko\newline
\myurl{https://orcid.org/0000-0002-1497-4338}\newline
e-mail: emaximenko@ipn.mx, egormaximenko@gmail.com

\medskip\noindent
Mario Alberto Moctezuma-Salazar\newline
\myurl{http://orcid.org/0000-0003-4579-0598}\newline
e-mail: m.a.mocte@gmail.com

\bigskip\noindent
Escuela Superior de F\'{i}sica y Matem\'{a}ticas\newline
Instituto Polit\'{e}cnico Nacional\newline
Ciudad de M\'{e}xico, Apartado Postal 07730\newline
Mexico

\end{document}